\documentclass[reqno,11pt,letterpaper]{amsart}

\usepackage{amsmath,amssymb,amsthm,graphicx,mathrsfs,url}
\usepackage[usenames,dvipsnames]{color}
\usepackage[colorlinks=true,linkcolor=Red,citecolor=Green]{hyperref}
\usepackage{amsxtra}
\usepackage{graphicx}
\usepackage{tikz}
\usetikzlibrary{positioning}
\relax

\numberwithin{equation}{section}

\newcommand{\vertiii}[1]{{\left\vert\kern-0.25ex\left\vert\kern-0.25ex\left\vert #1 
		\right\vert\kern-0.25ex\right\vert\kern-0.25ex\right\vert}}

\newcommand{\be}{\begin{equation}}
	\newcommand{\ee}{\end{equation}}

\newcommand{\ov}{\overline}

\newcommand{\simpletree}{\begin{tikzpicture}[scale=0.25]
		[
		level 1/.style={sibling distance=0.5cm},
		] 
		\coordinate (Root) [fill] circle (6pt)
		child { [fill] circle (6pt)
		}
		child  { [fill] circle (6pt)
		}
		child  { [fill] circle (6pt)
		}
		child  { [fill] circle (6pt)
		}
		child  { [fill] circle (6pt)
		}
		;
	\end{tikzpicture}
}   

\newcommand{\e}{{\mathrm{e} } }

\renewcommand{\Im}{\mathop{\rm Im}\nolimits}

\theoremstyle{plain}

\newtheorem{thm}{Theorem}[section]
\newtheorem{prop}{Proposition}[section]
\newtheorem{cor}[prop]{Corollary}

\newtheorem{lem}[prop]{Lemma}
\newtheorem{definition}[prop]{Definition}

\theoremstyle{definition}

\newtheorem{rem}[prop]{Remark}

\numberwithin{equation}{section}

\def\squarebox#1{\hbox to #1{\hfill\vbox to #1{\vfill}}}

\newcommand{\noi}{\noindent}

\usepackage{amsxtra}

\ifx\pdfoutput\undefined
\DeclareGraphicsExtensions{.pstex, .eps}
\else
\ifx\pdfoutput\relax
\DeclareGraphicsExtensions{.pstex, .eps}
\else
\ifnum\pdfoutput>0
\DeclareGraphicsExtensions{.pdf}
\else
\DeclareGraphicsExtensions{.pstex, .eps}
\fi
\fi
\fi

\title[Almost sure global nonlinear smoothing for the 2D NLS ]
{Almost sure global nonlinear smoothing for the 2D NLS}
\author{Chenmin Sun, Nikolay Tzvektov}
\address{}
\email{}

\usepackage{amssymb}
\usepackage{amsmath, amsthm, amsopn, amsfonts}

\def\11{{\rm 1~\hspace{-1.4ex}l} }
\def\R{\mathbb R}
\def\C{\mathbb C}
\def\Z{\mathbb Z}
\def\N{\mathbb N}

\def\T{\mathbb T}

\def\fk{\mathfrak}

\begin{document}

\begin{abstract}
	In this article, we prove an almost-sure global in time nonlinear smoothing effect for NLS on the two-dimensional torus. For deterministic data, this phenomenon was proved for the NLS on the circle by Erdo\u{g}an--Tzirakis, which remains unknown on multi-dimensional torus. 
	Our argument is based on a quantitative quasi-invariance of Gaussian measures with covariance operator $(1-\Delta)^{-s}$ for $s>2$.
\end{abstract}
	
	\maketitle 

\section{Introduction}

\subsection{Nonlinear smoothing effect for NLS}
In this article, we consider the defocusing nonlinear Schr\"odinger equation (NLS) on the two-dimensional torus $\T^2:=\R^2/(2\pi\Z)^2$:
\begin{align}\label{NLS}
	\begin{cases}
		& i\partial_t u + \Delta u = |u|^{2m-2}u,\quad (t,x) \in \R\times\T^2,\\
		& u|_{t=0}=\phi,
	\end{cases}
\end{align}
where $m\geq 2$ is an integer.
The equation \eqref{NLS} is a Hamiltonian system with conserved mass and energy:
\[
M[u]:=\int_{\T^2} |u|^2 \, dx,\qquad
H[u]:=\frac{1}{2}\int_{\T^2}|\nabla u|^2\,dx+\frac{1}{2m}\int_{\T^2}|u|^{2m}\,dx.
\]
It is known that \eqref{NLS} is globally well-posed on $H^{1}$ for any $m\geq 2$, thanks to the conserved energy and the Strichartz estimates in \cite{BoGAFA} or \cite{BGT-1}.

In this work, we study the smoothing property of the solutions to \eqref{NLS}. In sharp contrast with parabolic equations, the linear propagator of dispersive PDEs does not regularize the initial data; we are thus interested in the \emph{nonlinear smoothing effect} (NoSE) of solutions modulo the linear part under a suitable gauge transform.

To the authors' knowledge, (NoSE) can be used to extend fine properties of linear solutions, such as the Talbot effect (see \cite{Osko,KaRo,Ro} and references therein) and the pointwise convergence property (see \cite{DaKenig,DuGuLi,DuZh} for the most celebrated results), to nonlinear solutions. It is due to the point of view that the nonlinearity can be viewed as a qualitative perturbation (not small) of the linear evolution, in the setting of problems under consideration.
It is shown in \cite{ErTz} (and further extended in \cite{McC}) that the one-dimensional NLS on $\mathbb{T}$ exhibits (NoSE). Consequently, the authors proved the Talbot effect for NLS, which was observed numerically in \cite{Chen}.
More precisely, the nonlinear smoothing property Erdo\u{g}an--Tzirakis proved states as follows:
\begin{thm}\label{ET}\cite{ErTz}
	Let $\sigma>0$ and $T>0$. For any $\phi\in H^{\sigma}(\T)$, the solution of the cubic NLS  \textup{(\eqref{NLS} with $m=2$)} with initial data $\phi$ admits the decomposition
	\[
	u(t)=\e^{-\frac{it}{\pi}\int_{\T}|\phi|^2 \, dx}\cdot\big(\e^{it\partial_x^2}\phi + w(t)\big),
	\]
	on $t\in [-T,T]$, where $w(t)\in C([-T,T];H^{s_1}(\T))$ with $s_1<\sigma+\min(2\sigma,1/2)$.
\end{thm}

Let us also mention that the (NoSE) is used to prove exponential mixing for a randomly forced NLS on the 1D torus \cite{Xiang}. It serves as an important ingredient to prove the exponential asymptotic compactness property for the damped NLS.

For NLS on $\R^d$, the nonlinear smoothing effect was proved in \cite{Bobilinear,KeVa} as a consequence of the bilinear Strichartz estimate gaining anti-derivatives in the high-frequency part. In sharp contrast, due to the lack of dispersion on compact manifolds and worse bilinear Strichartz inequality (see \cite{BGT}), the (NoSE) for NLS with deterministic initial data remains unknown.

Nevertheless, it is reasonable to conjecture that such a phenomenon is not true for NLS on $\mathbb{T}^d$ with $d\geq 2$. We will support this conjecture in the context of the cubic NLS on $\T^2$ through the divergence of the second Picard iteration:
\begin{prop}\label{prop:lacksmoothing}
	Let $\sigma\geq 0$. There exist a bounded sequence $(\phi_{n})_{n\geq 1}$ in $H^{\sigma}(\T^2)$ and a countable set $E\subset\R$ such that for every $t\in\R\setminus E$,
	\[
	\lim_{n\to\infty} \Big\|\frac{1}{i}\int_0^t\e^{i(t-t')\Delta}\big(:|\e^{it'\Delta} \phi_n|^2\e^{it'\Delta}\phi_n : \big)\,dt'\Big\|_{H^{\sigma_1}(\T^2)}=+\infty,
	\]
	for any $\sigma_1>\sigma$, where
	\[
	:|v|^2v:\;=\;|v|^2v-\frac{1}{2\pi^2}\|v\|_{L^2}^2\,v
	\]
	is the renormalised cubic nonlinearity from the gauge transform.
\end{prop}
\begin{rem}
	From the analysis in \cite{ErTz}, Proposition \ref{prop:lacksmoothing} is not true for the 1D torus $\T$.
\end{rem}

\subsection{Nonlinear smoothing effect for the random initial data}
It turns out that the nonlinear smoothing effect is, however, a collective property when considering random initial data distributed according to centered Gaussian measures. More precisely, we assume that initial data are distributed according to the Gaussian probability measure $\mu_s$, formally defined as ``$\frac{1}{\mathcal{Z}}\mathrm{e}^{-\frac{1}{2}\|u\|_{H^s}^2}du$'', induced by the random Fourier series
\begin{align}\label{randominitial}
	\phi^{\omega}(x):=\sum_{k\in\Z^2}\frac{g_k(\omega)}{\langle k\rangle^s}\mathrm{e}^{ik\cdot x},
\end{align}
where $\langle k\rangle:=\sqrt{1+|k|^2}$. It is well-known that $\mu_s$ is supported on functions in $H^{\sigma}(\T^2)$ for all $\sigma<s-1$.
Our first main result is the following:
\begin{thm}\label{globalstructure}
	Let $s>2$, $\sigma<s-1$, $s_1<s-\frac{1}{2}$ and $T\geq 1$. For $\mu_s$-almost every $\phi\in H^{\sigma}(\T^2)$, the solution of the cubic NLS \textup{(\eqref{NLS} with $m=2$)} admits the decomposition
	\begin{align}\label{recenter2}
		u(t)=\e^{-\frac{it}{2\pi^2}\int_{\T^2}|\phi^{\omega}|^2dx}\cdot \big(\e^{it\Delta}\phi^{\omega}+ w(t)\big),
	\end{align}
	on $t\in[-T,T]$, where $w(t)\in C([-T,T];H^{s_1}(\T^2))$ with at most polynomial growth in $|t|\leq T$.
\end{thm}
\begin{rem}
	We note that \eqref{recenter2} is also true for $s=1$ (\cite{Bo96}), thanks to the equivalence between $\mu_1$ and the invariant Gibbs measure. The case $s>1$
is proved \cite{CLSta}, but only for small $T>0$. It is an interesting question to fill the gap $1<s\leq 2$ in Theorem \ref{globalstructure}.
\end{rem}
\begin{rem}
The analogue of Theorem \ref{globalstructure} for NLS with any $m\geq 2$ should also hold under a different gauge with the phase factor $\e^{-\frac{im}{(2\pi)^2}\int_0^t\|u(t')\|_{L^{2m}}^{2m}dt' }$ by our argument. Since the local-in-time statement (analogue of Proposition \ref{localstructure}) does not seem to be written in the literature, we prefer not stating the general case as a theorem.
\end{rem}

\subsection{Quantitative quasi-invariance}
Note that the NLS flow $\Phi_t$ of \eqref{NLS} is a globally defined flow on $H^1(\T^2)$. In particular, when $s>2$, $\Phi_t$ is a well-defined flow on the support of the measure $\mu_s$. The proof of Theorem \ref{globalstructure} is deduced from a local nonlinear smoothing effect (Proposition \ref{localstructure}), and more importantly, the quantitative quasi-invariance property of the Gaussian measure $\mu_s$ transported by the flow $\Phi_t$.

Initiated by the second author in \cite{Tzsigma}, intensive research activities have been conducted for proving quasi-invariance of Gaussian measures transported by flows of dispersive PDEs. Among an exhaustive list of literature, we only mention \cite{BuTho,DeTsu,FoSe,FoTo,GLTz,Kne1,Kne2,Kne3,OhSeo,OhSoTz,OTz4NLS,PTV,SunTzCritical} for NLS type equations. These results are mostly for one-dimensional models except for the previous work \cite{SunTzCritical} of the authors for the 3D energy-critical NLS.

Our second main result in this article is the following \emph{quantitative  quasi-invariance} property of NLS with any integer-valued defocusing nonlinearity:
\begin{thm}\label{thm:main}
	Assume that $s>2$ and $m\in \N$, $m\geq 2$. Then for any $t\in\R$, the transported Gaussian measure $(\Phi_t)_*\mu_s$ is absolutely continuous with respect to $\mu_s$. More precisely, for any $\lambda \geq 1$, there exists a family of weighted measures
	\[
	d\rho_{s,\lambda }(u)=\mathbf{1}_{H[u]\leq \lambda}\cdot \mathrm{e}^{-R_{s}(u)}\,d\mu_s(u)
	\]
	such that
	\begin{itemize}
		\item $\mathbf{1}_{H[u]\leq \lambda}\cdot \mathrm{e}^{| R_{s}(u)|}\in L^p(d\mu_s)$ for any $1\leq p<\infty$ and $\lambda> 0$.
		\item The density $f_{\lambda }(t,\cdot)$ of the transported measure $(\Phi_t)_*d\rho_{s,\lambda }$ belongs to $L^p(d\rho_{s,\lambda })$ for any $1\leq p<\infty$. More precisely, there exist $C(p,\lambda)>0$ and an absolute constant $C_0>0$ such that
		\[
		\|f_{\lambda}(t,\cdot)\|_{L^p(d\rho_{s,\lambda})}\leq C(p,\lambda)\,\mathrm{e}^{\,C(p,\lambda)(1+|t|^{C_0})}.
		\]
	\end{itemize}
\end{thm}
\begin{rem}
After this manuscript was essentially completed, we became aware of a work by Tolomeo--Visciglia \cite{ToVi} which obtains a similar result of Theorem \ref{thm:main} by using a different approach, based on  physical space analysis to build modified energies.
\end{rem}
\begin{rem}
	For large values of $s$, the quasi-invariance part of Theorem \ref{thm:main} can be deduced directly from the same argument in \cite{SunTzCritical}.
	The statements here are quantitative, in the sense that the Radon--Nikodym density of the weighted Gaussian measure (truncated by a conservation law) is in any $L^p(d\mu_s)$ for $p<\infty$. By mimicking Bourgain's measure-invariant argument \cite{Bo96}, such a property can be applied to obtain various dynamical consequences of the solutions. See \cite{FoTo,Kne3} about the application of the quantitative quasi-invariance to globalize solutions in the low-regularity setting.
\end{rem}
\begin{rem}
	Even for large value of $s$, the quantitative quasi-invariance is more difficult to obtain, compared to the qualitative quasi-invariance property. From the current technology, for proving the quantitative quasi-invariance, we have to truncate the measure with some conservation law, while we can truncate the stronger well-posedness norms for proving qualitative quasi-invariance.
\end{rem}

\begin{rem}
Let us also mention two other applications in the literature. In the recent work \cite{BLTV}, the quantitative quasi-invariance of Gaussian measures was used 
in order to prove low regularity well-posedness results for the binormal flow. In \cite{BuTho}, such an idea was used to prove the almost sure scattering of the NLS on $\R$.
\end{rem}

Our proof of Theorem \ref{thm:main} relies on a methodology similar to the one developed in \cite{SunTzCritical}. As we are dealing with solutions with relatively low regularity (though still above $H^1$), more attention is needed to control the interaction of solutions in different scales.
The constraint $s>2$ here is only to ensure the existence of global flow on the support of $\mu_s$, and the legality of the truncation of the energy.
We believe that our result can be extended to some $s$ slightly smaller than $2$ (depending on the degree of nonlinearity $m$), similar to \cite{Kne3} for the one-dimensional NLS. In another direction, it is likely that Theorem \ref{globalstructure} and Theorem \ref{thm:main} could be extended for the cubic NLS on $\T^3$ by our method.

	\subsection*{Acknowledgments}
	This work is partially supported by the ANR project Smooth ANR-22CE40-0017.
	The authors would like to thank Yuxuan Chen and Shengquan Xiang for interesting discussion which motivates the main result, Theorem \ref{globalstructure}.   
	The authors would like also to thank Yu Deng for conversation at an early stage of the study for the quasi-invariance for multi-dimensional NLS. 



\section{Modified energy and the weighted Gaussian measure}\label{Sec:modifiedenergy}

\subsection{Normal form and the modified energy}
Let $m\in\N$, $m\geq 2$ and $s\geq 2$. To construct appropriate weighted measures, we introduce a modified energy functional. For a smooth solution $u(t)$ of \eqref{NLS}, define the new unknown obtained by factoring out the linear flow:
\[
v(t)=\e^{-it\Delta}u(t).
\]
Expanding $v(t)$ in the Fourier series,
\[
v(t,x)=\sum_{k\in\Z^2} v_k(t)\,\mathrm{e}^{ik\cdot x},
\]
we see that $v_k(t)$ solves
\begin{align}\label{NLSk}
	i\partial_t v_k(t)=\sum_{\substack{k_1-k_2+\cdots+k_{2m-1}=k}}\e^{-it\Omega(\vec{k})}\,v_{k_1}(t)\,\ov{v}_{k_2}(t)\cdots v_{k_{2m-1}}(t),
\end{align}
where
\[
\Omega(\vec{k})=\sum_{j=1}^{2m-1}(-1)^{j-1}|k_j|^2-|k|^2
\]
is the resonant function. 
To construct the modified energy, it is convenient to use the following equivalent Sobolev norm for $s\geq 0$:
\begin{align}\label{def:Hsequiv}
	\vertiii{f}_{H^s(\T^2)}^2:=\sum_{k\in\Z^2}\big(1+|k|^{2s}\big)\,|\widehat{f}(k)|^2. 
\end{align}
A straightforward computation using the symmetry of the indices yields
\begin{align}\label{deriveeHs}
	\frac{1}{2}\frac{d}{dt}\,\vertiii{v(t)}_{H^s}^2
	=-\frac{1}{2m}\Im\!\sum_{\substack{k_1-k_2+\cdots -k_{2m}=0}}\!
	\psi_{2s}(\vec{k})\,\e^{-it\Omega(\vec{k})}\,v_{k_1}\ov{v}_{k_2}\cdots \ov{v}_{k_{2m}},
\end{align}
where (abusing notation slightly)
\[
\psi_{2s}(\vec{k})=\sum_{j=1}^{2m}(-1)^{j-1}|k_j|^{2s},\qquad
\Omega(\vec{k})=\sum_{j=1}^{2m}(-1)^{j-1}|k_j|^2.
\]
We split the right-hand side according to the resonance set $\{\Omega(\vec{k})=0\}$:
\begin{align}
	\frac{1}{2}\frac{d}{dt}\,\vertiii{v(t)}_{H^s}^2
	=&-\frac{1}{2m}\Im\!\sum_{\substack{k_1-k_2+\cdots-k_{2m}=0\\ \Omega(\vec{k})=0}}
	\psi_{2s}(\vec{k})\,\e^{-it\Omega(\vec{k})}\,v_{k_1}\ov{v}_{k_2}\cdots \ov{v}_{k_{2m}}\notag\\
	&-\frac{1}{2m}\Im\!\sum_{\substack{k_1-k_2+\cdots-k_{2m}=0\\ \Omega(\vec{k})\neq 0}}
	\frac{\psi_{2s}(\vec{k})}{-i\,\Omega(\vec{k})}\,\partial_t\!\Big(\e^{-it\Omega(\vec{k})}\,v_{k_1}\ov{v}_{k_2}\cdots\ov{v}_{k_{2m}}\Big)\notag\\
	&+\frac{1}{2m}\Im\!\sum_{\substack{k_1-k_2+\cdots-k_{2m}=0\\ \Omega(\vec{k})\neq 0}}
	\frac{\psi_{2s}(\vec{k})}{-i\,\Omega(\vec{k})}\,\e^{-it\Omega(\vec{k})}\,\partial_t\!\big(v_{k_1}\ov{v}_{k_2}\cdots\ov{v}_{k_{2m}}\big).\label{deriveeHs1}
\end{align}
Motivated by this, define the modified energy
\begin{align}\label{modifiedenergy}
	\mathcal{E}_{s,t}(v):=\frac{1}{2}\|v\|_{H^s(\T^2)}^2+\mathcal{R}_{s,t}(v),
\end{align}
where
\[
\mathcal{R}_{s,t}(v):=\frac{1}{2m}\Im \sum_{\substack{k_1-k_2+\cdots-k_{2m}=0\\ \Omega(\vec{k})\neq 0}}
\frac{\psi_{2s}(\vec{k})}{-i\,\Omega(\vec{k})}\,\e^{-it\Omega(\vec{k})}\,v_{k_1}\ov{v}_{k_2}\cdots \ov{v}_{k_{2m}}.
\]
Returning to $u$, observe that
\begin{equation}\label{Est}
	\mathcal{E}_{s,t}(v)= E_{s}(u):=\frac{1}{2}\,\vertiii{u}_{H^s(\T^2)}^2+R_{s}(u),
\end{equation}
with
\begin{align}\label{Rst}
	R_{s}(u):=\frac{1}{2m}\Im \sum_{\substack{k_1-k_2+\cdots-k_{2m}=0\\ \Omega(\vec{k})\neq 0}}
	\frac{\psi_{2s}(\vec{k})}{-i\,\Omega(\vec{k})}\,\widehat{u}_{k_1}\ov{\widehat{u}}_{k_2}\cdots \ov{\widehat{u}}_{k_{2m}}\,.
\end{align}
Then, from \eqref{deriveeHs1}, \eqref{NLSk} and the symmetry of the indices, we obtain
$$ \frac{d}{dt}\,\mathcal{E}_{s,t}(v)=\frac{d}{dt}E_s(u)=\mathrm{I}+\mathrm{II}+\mathrm{III},
$$
where
\begin{align}\label{dmodifiedenergy}
\mathrm{I}
	:=&-\frac{1}{2m}\Im\!\sum_{\substack{k_1-k_2+\cdots-k_{2m}=0\\ \Omega(\vec{k})=0}}
	\psi_{2s}(\vec{k})\,\e^{-it\Omega(\vec{k})}\,v_{k_1}\ov{v}_{k_2}\cdots \ov{v}_{k_{2m}}\notag,\\
	\mathrm{II}:=&\frac{1}{2}\Im\!\sum_{\substack{k_1-k_2+\cdots-k_{2m}=0\\ \Omega(\vec{k})\neq 0}}
	\frac{\psi_{2s}(\vec{k})}{\Omega(\vec{k})}\sum_{\substack{k_1=p_1-p_2+\cdots+p_{2m-1}}}
	\e^{-it\big(\Omega(\vec{k})+\Omega(\vec{p})\big)}\,v_{p_1}\ov{v}_{p_2}\cdots v_{p_{2m-1}}\ov{v}_{k_2}\cdots \ov{v}_{k_{2m}}\notag,\\
	\mathrm{III}:=&-\frac{1}{2}\Im\!\sum_{\substack{k_1-k_2+\cdots -k_{2m}=0\\ \Omega(\vec{k})\neq 0}}
	\frac{\psi_{2s}(\vec{k})}{\Omega(\vec{k})}\sum_{\substack{k_2=q_1-q_2+\cdots +q_{2m-1}}}
	\e^{-it\big(\Omega(\vec{k})-\Omega(\vec{q})\big)}\,v_{k_1}\ov{v}_{q_1}\cdots \ov{v}_{q_{2m-1}}v_{k_3}\cdots \ov{v}_{k_{2m}},
\end{align}
where
\[
\Omega(\vec{p})=\sum_{j=1}^{2m-1}(-1)^{j-1}|p_j|^2-|k_1|^2,\qquad
\Omega(\vec{q})=\sum_{j=1}^{2m-1}(-1)^{j-1}|q_j|^2-|k_2|^2.
\]

\subsection{The weighted measure}
Using the modified energy, we define, for $\lambda\geq 1$, the weighted Gaussian measure
\begin{align}\label{rhostr}
	d\rho_{s,\lambda }(u):=\mathbf{1}_{\{H[u]\leq \lambda\}}\,\e^{-R_{s}(u)}\,d\mu_s(u).
\end{align}
This can be made rigorous by studying the truncated system
\begin{align}\label{NLS-N}
	\begin{cases}
		&i\partial_tu_N+\Delta u_N=\pi_N\!\big(|\pi_Nu_N|^{2m-2}\,\pi_Nu_N\big),\\
		&u_N|_{t=0}=\phi,
	\end{cases}
\end{align}
where $\widehat{\pi_N f}(k)=\mathbf{1}_{\{|k|\leq N\}}\widehat{f}(k)$. Note that \eqref{NLS-N} is a Hamiltonian system with Hamiltonian
\[
H_N[u]:=\frac{1}{2}\int_{\T^2}|\nabla \pi_Nu|^2\,dx+\frac{1}{2m}\int_{\T^2}|\pi_Nu|^{2m}\,dx.
\]
Denote the flow of \eqref{NLS-N} by $\Phi_t^N$.

Define accordingly the functionals $\mathcal{E}_{s,t,N}$, $E_{s,N}$, $\mathcal{R}_{s,t,N}$, $R_{s,N}$ by replacing $\widehat{u}_k, v_k$ with $\widehat{u}_k\mathbf{1}_{\{|k|\leq N\}}$, $v_k\mathbf{1}_{\{|k|\leq N\}}$ in the formulas of the previous subsection. Let $\mu_{s,N}$ be the Gaussian measure on the finite-dimensional space $\pi_NL^2$ induced by the truncated random Fourier series
\[
\pi_N\phi^{\omega}(x):=\sum_{|k|\leq N}\frac{g_k^{\omega}}{\langle k\rangle^s}\,\e^{ik\cdot x}.
\]
Note that $\mu_s=\mu_{s,N}\otimes \mu_{s,N}^{\perp}$, where $\mu_{s,N}^{\perp}$ is the Gaussian measure on $\pi_N^{\perp}L^2$ induced by $(\mathrm{Id}-\pi_N)\phi^{\omega}(x)$.

\begin{prop}[Construction of the weighted measure]\label{weightedmeasure}
	Let $s>2$ and $\lambda\geq 1$. For any $p\in[1,\infty)$, there exists a constant $C(p,s,\lambda)>0$, independent of $N$, such that for every $N\in\N$,
	\[
	\big\|\mathbf{1}_{\{H_N[u]\leq \lambda\}} \,\e^{|R_{s,N}(u)|}\big\|_{L^p(d\mu_s)}\leq C(p,s,\lambda).
	\]
	Moreover,
	\[
	\lim_{N\to\infty}\Big\|
	\mathbf{1}_{\{H_N[u]\leq \lambda\}}\,\e^{-R_{s,N}(u)}
	-\mathbf{1}_{\{H[u]\leq \lambda\}}\,\e^{-R_{s}(u)}
	\Big\|_{L^p(d\mu_s)}=0.
	\]
\end{prop}

The proposition defines the convergent sequence of weighted measures
\[
d\rho_{s,\lambda ,N}(u):=\mathbf{1}_{\{H_N[u]\leq \lambda\}}\,\e^{-R_{s,N}(u)}\,d\mu_{s,N}(u)\otimes d\mu_{s,N}^{\perp}(u),
\]
whose \emph{densities} converge $\mu_s$-a.s. and in $L^p(d\mu_s)$ to the density of $\rho_{s,\lambda}$. The weighted measure $\rho_{s,\lambda}$ is equivalent to the truncated Gaussian measure $\mu_{s,\lambda}:=\mathbf{1}_{\{H[u]\leq \lambda\}}\,\mu_s$.

\begin{prop}[Weighted energy estimate]\label{energyestimate}
	Let $s>2$, $\lambda \geq 1$, and $N\in\N\cup\{\infty\}$. Define
	\[
	Q_{N}(u):=\Big(\frac{d}{dt}E_{s,N}\big(\Phi_{t}^N(u)\big) \Big)\Big|_{t=0}.
	\]
	Then there exist constants $C(s,\lambda)>0$ and $\epsilon_0\in(0,1)$, independent of $N$, such that for all $p\in[2,\infty)$,
	\[
	\big\|\mathbf{1}_{\{H_N[u]\leq \lambda\}}\; Q_{N}(u)\big\|_{L^p(d\mu_s)}\leq C(s,\lambda,\epsilon_0)\,p^{1-\epsilon_0}.
	\]
	By Proposition \ref{weightedmeasure} and Cauchy--Schwarz, we also have, for all $p\in[2,\infty)$,
	\[
	\big\|\mathbf{1}_{\{H_N[u]\leq \lambda\}}\; Q_{N}(u)\big\|_{L^p(d\rho_{s,\lambda,N})}\leq C(s,\lambda,\epsilon_0)\,p^{1-\epsilon_0}.
	\]
\end{prop}

\subsection{Quasi-invariance and $L^p$ bound of the transported density}

As a consequence of Proposition \ref{weightedmeasure}, Proposition \ref{energyestimate} and Lemma \ref{abstract}, we now sketch the proof of Theorem \ref{thm:main}.

Following a standard argument (see the proof of Lemma~3.3 in \cite{SunTzCritical}), for any $N\in\N$ and any measurable $A$,
\[
\frac{d}{dt}\,\rho_{s,\lambda,N}\big((\Phi_t^N)^{-1}(A)\big)
\;\le\; \big\|\mathbf{1}_{\{H_N[u]\leq \lambda\}}\; Q_N(u) \big\|_{L^p(d\rho_{s,\lambda,N})}\,
\big[\rho_{s,\lambda,N}\big((\Phi_t^{N})^{-1}(A)\big) \big]^{\,1-\frac{1}{p}},
\]
for all $p\in[1,\infty)$. By Proposition \ref{energyestimate},
\begin{align}\label{ineq:differential}
	\frac{d}{dt}\,\rho_{s,\lambda,N}\big((\Phi_t^N)^{-1}(A)\big)\leq C(s,\lambda,\epsilon_0)\,p^{1-\epsilon_0}\,
	\big[\rho_{s,\lambda,N}\big((\Phi_t^{N})^{-1}(A)\big) \big]^{\,1-\frac{1}{p}}.
\end{align}
Applying Lemma \ref{abstract}, we obtain, for any fixed $\epsilon_0>0$,
\begin{align}\label{transportrho}
	\rho_{s,\lambda,N}\big((\Phi_t^N)^{-1}(A)\big)\leq C(s,\lambda,\epsilon_0)\,\rho_{s,\lambda,N}(A)^{\,1-\epsilon_0},
\end{align}
and, for any $p\in[1,\infty)$,
\begin{equation}\label{Lp:fNtu}
	\sup_{|t|\le T}\Big\|\frac{d\,(\Phi_{t}^N)_{*}\rho_{s,\lambda,N}}{d\rho_{s,\lambda,N}} \Big\|_{L^p(d\rho_{s,\lambda,N})}
	\ \le\ C(s,\lambda,\epsilon_0,p)\,\exp\!\Big(C(s,\lambda,\epsilon_0,p)\big(1+T^{2/\epsilon_0}\big)\Big).
\end{equation}
To pass to the limit $N\to\infty$, we invoke the standard stability/Cauchy theory for \eqref{NLS} at regularity $s>2$ (deterministic and subcritical), as in \cite{Tzsigma} and related references; we omit the routine details here.


	
	\section{Tree notations and preliminary estimates}

	The proof of Proposition \ref{weightedmeasure} and Proposition \ref{energyestimate} relies on estimates for certain multilinear expressions of Gaussian random variables stemming from the modified energy \eqref{modifiedenergy} and its time derivative \eqref{dmodifiedenergy}. It is convenient to index these expressions using trees, inspired by \cite{DH1}.
	
	\subsection{Simple trees and expanded branching trees}
	
	We first define a \emph{simple tree} $\mathcal{T}$ of size $n(\mathcal{T})=2m_0-1$ for some $m_0\in\mathbb{N}$, with a root $\fk{r}$ and $2m_0-1$ leaves $\mathcal{L}$ as its children. A simple tree looks like $\simpletree$. 
	Label the leaves from left to right by $\fk{n}_j$ and set the signs by $\iota_{\fk{r}}=+1$ and $\iota_{\fk{n}_j}=(-1)^{j-1}$ for $j=1,\dots,2m_0-1$. 
	An assignment $(k_{\fk{n}}:\fk{n}\in\mathcal{T})$ is called a \emph{decoration} (following \cite{DH1}) if it satisfies
	\[
	\sum_{\fk{n}\in\mathcal{L}}\iota_{\fk{n}}k_{\fk{n}}=k_{\fk{r}}.
	\]
	
	Next we define an \emph{expanded branching tree} $\mathcal{T}$: its vertices are \emph{branching nodes} $\mathcal{N}$ (nodes with children) and \emph{leaves} $\mathcal{L}$ (nodes without children), and satisfy:
	\begin{itemize}
		\item There exists a unique root $\fk{r}\in\mathcal{N}$ that has no ancestor, has an odd number of children, and a \emph{partner leaf} $\fk{r}'\in\mathcal{L}$. Any other node $\fk{n}\in\mathcal{N}\setminus\{\fk{r}\}$ has exactly one parent, has an odd number of children, and has no partner leaf.
	\end{itemize}
	
	\begin{center}
		\begin{tikzpicture}[scale=0.5,
			level 1/.style={sibling distance=1cm},
			level 2/.style={sibling distance=1cm}, level distance=1cm]
			\draw (1,0) circle (3.0pt);
			\draw (0,0).. controls (0.3,0.3) and (0.8,0.3) .. (0.95,0);
			\node at (0,0) [left] {$\mathfrak{r}$};
			\node at (1,0) [right] {$\mathfrak{r}'$};
			\coordinate (root) {} [fill] circle (3.0pt)
			child { [fill] circle (3.0pt)
				child { [fill] circle (3.0pt) }
				child { circle (3.0pt) }
				child { [fill] circle (3.0pt) }
			}
			child { circle (3.0pt) }
			child { [fill] circle (3.0pt) }
			child { circle (3.0pt) }
			child { [fill] circle (3.0pt) };
			\node at (root-1)[left] {$\mathfrak{n}_0$};
			\node at (0,-5) {An expanded branching tree of scale 2};
		\end{tikzpicture}
	\end{center}
	
	For an expanded branching tree, we use the following notions:
	\begin{itemize}
		\item The \emph{scale} $s(\mathcal{T})$ is the number of branching nodes $\#\mathcal{N}$. 
		The \emph{size} $n(\mathcal{T})$ is the number of leaves $\#\mathcal{L}$ (the partner leaf $\mathfrak{r}'$ is also a leaf). 
		For $\fk{n}\in\mathcal{N}\setminus\{\fk{r}\}$, denote by $\mathcal{L}_{\fk{n}}$ the children of $\fk{n}$ that are leaves of $\mathcal{T}$. 
		For the root $\fk{r}$, let $\mathcal{L}_{\fk{r}}$ be the set of children of $\fk{r}$ that are leaves, \emph{together with the partner leaf $\fk{r}'$}; then
		\[
		\#\mathcal{L}=\sum_{\fk{n}\in\mathcal{N}}\#\mathcal{L}_{\fk{n}}.
		\]
		The first generation (children of $\fk{r}$ together with the partner leaf) is sometimes labeled $\{\fk{r}_1,\dots,\fk{r}_{2\ell_0}\}$ from left to right, with $\fk{r}_{2\ell_0}=\fk{r}'$. 
		\vspace{0.3cm}
		
		\item For a leaf $\fk{l}$, denote its unique parent node by $\fk{p}(\fk{l})$.
		
		\vspace{0.3cm}
		\item Fix signs $\iota_{\fk{n}}\in\{+1,-1\}$ inductively as follows:
		\begin{itemize}
			\item Set $\iota_{\fk{r}}=+1$ and $\iota_{\fk{r}'}=-1$. 
			If the children of $\fk{r}$ are ordered from left to right as $\fk{n}_1,\dots,\fk{n}_{2\ell_0-1}$, define their signs \emph{alternatingly} by $\iota_{\fk{n}_{2j-1}}=+1$ and $\iota_{\fk{n}_{2j}}=-1$ for $j=1,2,\dots,\ell_0$.
			\item For any non-root node $\fk{n}\in\mathcal{N}$, if its children from left to right are $\fk{n}_1,\dots,\fk{n}_{2\ell-1}$, define their signs \emph{alternating}, starting from $\iota_{\fk{n}_1}=\iota_{\fk{n}}$.
		\end{itemize}
		
		\vspace{0.3cm}
		\item A \emph{decoration} $(k_{\fk{n}}:\fk{n}\in\mathcal{T})$ is an assignment $k_{\fk{n}}\in\mathbb{Z}^2$ for each $\fk{n}\in\mathcal{T}$ satisfying $k_{\fk{r}}=k_{\fk{r}'}$ and, for any $\fk{n}\in\mathcal{N}$ with children $\fk{n}_1,\dots,\fk{n}_{2\ell_0-1}$ (left to right),
		\[
		k_{\fk{n}}=\sum_{j=1}^{2\ell_0-1}\iota_{\fk{n}_j}\,k_{\fk{n}_j}.
		\]
		
		\vspace{0.3cm}
		\item Resonant functions and weights. 
		For the root $\fk{r}$ with first-generation nodes $\fk{r}_1,\dots,\fk{r}_{2\ell_0-1}$ (left to right), define
		\[
		\Omega(\vec{k}_{\fk{r}}):=\sum_{j=1}^{2\ell_0-1}\iota_{\fk{r}_j}|k_{\fk{r}_j}|^2-\iota_{\fk{r}}|k_{\fk{r}}|^2
		=\sum_{\fk{l}\in\mathcal{L}_{\fk{r}}}\iota_{\fk{l}}|k_{\fk{l}}|^{2},
		\]
		\[
		\psi_{2s}(\vec{k}_{\fk{r}}):=\sum_{j=1}^{2\ell_0-1}\iota_{\fk{r}_j}|k_{\fk{r}_j}|^{2s}-\iota_{\fk{r}}|k_{\fk{r}}|^{2s}
		=\sum_{\fk{l}\in\mathcal{L}_{\fk{r}}}\iota_{\fk{l}}|k_{\fk{l}}|^{2s}.
		\]
		For each non-root node $\fk{n}$ with children $\fk{n}_1,\dots,\fk{n}_{2\ell_0-1}$ (left to right), set
		\[
		\Omega(\vec{k}_{\fk{n}}):=\iota_{\fk{n}}\Big(\sum_{j=1}^{2\ell_0-1}\iota_{\fk{n}_j}|k_{\fk{n}_j}|^2-\iota_{\fk{n}}|k_{\fk{n}}|^2\Big)
		=\iota_{\fk{n}}\Big(\sum_{\fk{l}\in\mathcal{L}_{\fk{n}}}\iota_{\fk{l}}|k_{\fk{l}}|^{2}-\iota_{\fk{n}}|k_{\fk{n}}|^2\Big).
		\]
	\end{itemize}
	
	\noindent\emph{Remark.} We do not require each node to have the same number of children, since our algorithm may delete paired leaves to produce a reduced tree.
	
	\medskip
	\noindent\textbf{Quasi-order for leaves.}
	Let $\mathcal{T}$ be a scale-$2$ expanded branching tree with $\mathcal{N}=\{\fk{r},\fk{n}_0\}$. 
	Given dyadic numbers $(N_{\fk{n}}:\fk{n}\in\mathcal{T})$, we say that $\mathcal{T}$ is \emph{adapted} to $(N_{\fk{n}})$ if every decoration $(k_{\fk{n}})$ also satisfies $N_{\fk{n}}\le|k_{\fk{n}}|<2N_{\fk{n}}$ for all $\fk{n}\in\mathcal{T}$. 
	Assume that $\#\mathcal{L}=2m_0$. Denote by $N_{\fk{l}_1}\ge N_{\fk{l}_2}\ge\cdots\ge N_{\fk{l}_{2m_0}}$ a non-increasing rearrangement of the leaves’ dyadic sizes $(N_{\fk{l}}:\fk{l}\in\mathcal{L})$. Note that for adjacent indices $\fk{l}_j,\fk{l}_{j+1}$ it is not necessary that $|k_{\fk{l}_j}|\ge |k_{\fk{l}_{j+1}}|$, but they do satisfy $|k_{\fk{l}_j}|\ge 2\,|k_{\fk{l}_{j+1}}|$.
	
	\noindent\textbf{Example.} If $\mathcal{L}=\{2,3,4,5,6,7\}$ and $N_2\ge N_4\ge N_5\ge N_7\ge N_6\ge N_3$, then 
	\[
	\fk{l}_1=2,\quad \fk{l}_2=4,\quad \fk{l}_3=5,\quad \fk{l}_4=7,\quad \fk{l}_5=6,\quad \fk{l}_6=3.
	\]
	More generally, for a subset $J\subset \mathcal{L}$, write $\fk{l}_1^J,\dots,\fk{l}_{\#J}^J$ for a permutation of indices in $J$ so that
	\[
	N_{\fk{l}_1^J}\ge N_{\fk{l}_2^J}\ge\cdots\ge N_{\fk{l}_{\# J}^J }.
	\]
	
	\medskip
	\noindent\textbf{Pairing.}
	A pairing $\mathcal{P}=(X,Y,\sigma)$ consists of two non-empty subsets $X,Y\subset\mathcal{L}$ with $\#X=\#Y$ and $X\cap Y=\emptyset$, together with a bijection $\sigma:X\to Y$. 
	Giving an expanded branching tree $\mathcal{T}$ with a pairing $\mathcal{P}$ means we pair leaves $(\fk{l},\sigma(\fk{l}))$ for each $\fk{l}\in X$.  
	A decoration $(k_{\fk{n}})$ for a paired tree must additionally satisfy $k_{\fk{l}}=k_{\sigma(\fk{l})}$ and $\iota_{\fk{l}}+\iota_{\sigma(\fk{l})}=0$ for each $\fk{l}\in X$. 
	We sometimes say leaves $(k_{\fk{i}},k_{\fk{j}})$ are paired if $k_{\fk{i}}=k_{\fk{j}}$ and $\iota_{\fk{i}}+\iota_{\fk{j}}=0$. 
	If $\mathcal{T}$ is adapted to the dyadic sizes $(N_{\fk{n}})$, then paired leaves $\fk{j},\sigma(\fk{j})$ impose $N_{\fk{j}}\sim N_{\sigma(\fk{j})}$.
	
	\medskip
	\noindent\textbf{Energy multilinear form.}
	Let $\mathcal{T}$ be an expanded branching tree of scale $s(\mathcal{T})\le 2$ and size $n(\mathcal{T})=2n_0$. 
	For input functions $a^{(1)},\dots,a^{(2n_0)}\in h^\sigma(\mathbb{Z}^2)$, define the \emph{energy multilinear form}
	\begin{align}\label{Te}
		\mathcal{T}_e(a^{(1)},\dots,a^{(2n_0)}) 
		&:= \sum_{(k_{\fk{n}}:\fk{n}\in\mathcal{T})}
		\frac{\psi_{2s}(\vec{k}_{\fk{r}})}{\Omega(\vec{k}_{\fk{r}})}\,\mathbf{1}_{\Omega(\vec{k}_{\fk{r}})\neq 0}
		\prod_{j=1}^{2n_0}\bigl(a^{(\fk{l}_j)}_{k_{\fk{l}_j}}\bigr)^{\iota_{\fk{l}_j}},
		\qquad \text{if } s(\mathcal{T})=2,\\
		\mathcal{T}_e(a^{(1)},\dots,a^{(2n_0)}) 
		&:= \sum_{(k_{\fk{n}}:\fk{n}\in\mathcal{T})}
		\psi_{2s}(\vec{k}_{\fk{r}})\,\mathbf{1}_{\Omega(\vec{k}_{\fk{r}})=0}
		\prod_{j=1}^{2n_0}\bigl(a^{(\fk{l}_j)}_{k_{\fk{l}_j}}\bigr)^{\iota_{\fk{l}_j}},
		\qquad \text{if } s(\mathcal{T})=1.
	\end{align}
	In particular, the last two terms on the right-hand side of \eqref{dmodifiedenergy} are of the form $\Im\,\mathcal{T}_e(v)$, where $\mathcal{T}_e(v):=\mathcal{T}_e(v,\dots,v)$.

		\subsection{Main steps of the modified energy estimate}
	
Now we outline the scheme of the proof for Proposition \ref{energyestimate} for estimating an expanded branching tree $\mathcal{T}$ of scale $2$. Denote $\fk{r}$ the root and $\fk{n}_0$ the other unique non-root branching node. In this subsection, we assume that $\# \mathcal{L}_{\fk{r}}=\#\mathcal{L}_{\fk{n}_0}=n_0=2m-1$. 
	
	Following \cite{SunTzCritical}, for a small parameter $\theta=\frac{1}{50(n_0+1)}$, we define
	$$ \Lambda:=\bigcup_{\substack{\fk{l}'\in\mathcal{L}_{\fk{r}},\fk{l}''\in\mathcal{L}_{\fk{n}_0}\\
			\iota_{\fk{l}'}+\iota_{\fk{l}''}=0 } }\Lambda_{\fk{l}',\fk{l}''},
	$$
	where
	\begin{align}\label{Lambda} 
		\Lambda_{\fk{l}',\fk{l}''}:=\Big\{(k_{\fk{n}}:\fk{n}\in\mathcal{T}):\Omega(\vec{k}_{\fk{r}})\neq 0,\; k_{\fk{l}'}=k_{\fk{l}''},\; \sum_{\fk{l}\in\mathcal{L}_{\fk{r}},\;\fk{l}\neq\fk{l}' }|k_{\fk{l}}|\leq |k_{\fk{n}_0}|^{\theta}+|k_{\fk{l'}}|^{\theta},\; \sum_{\fk{l}\in\mathcal{L}_{\fk{n}_0},\;\fk{l}\neq \fk{l}'';; }|k_{\fk{l}}|\leq |k_{\fk{n}_0}|^{\theta}+|k_{\fk{l}''}|^{\theta}  \Big\},
	\end{align}
	 is the set of decorations with a single high-frequency cross-pairing. 
	We write the cross-pairing singular contribution in $\mathcal{T}$, corresponding to the paired leaves $\fk{l}',\fk{l''}$ as
	\begin{align}\label{ST} 
		\mathcal{S}_{\fk{l}',\fk{l}''}(\mathcal{T})(v):=\sum_{ (k_{\fk{n}})\in \Lambda_{\fk{l}',\fk{l}''}  }
		\frac{\psi_{2s}(\vec{k}_{\fk{r}})}{\Omega(\vec{k}_{\fk{r}})}\prod_{\fk{l}\in\mathcal{L}}v_{k_{\fk{l}}}^{\iota_{\fk{l}}}.
	\end{align}
	It was already observed in \cite{SunTzCritical} that, when estimating each individual term $\mathcal{S}_{\fk{l}',\fk{l}''}(\mathcal{T})(v)$, the paring between two leaves of top frequencies prevents us to gain from the square-root cancellation of independent random variables. The key observation is that when taking the imaginary part, we cancel the most singular contribution in $\mathcal{S}_{\fk{l}',\fk{l}''}(\mathcal{T})(v)$, which allows us to conclude by purely deterministic argument. We remark that the two-dimensional case here is simpler than the three-dimensional case in \cite{SunTzCritical}, where we need more cancellations between different cross-pairing structures.

	Below we summarize the analysis to be done in the following sections:
	\begin{itemize}
		\item In Section \ref{Sec:Basicestimates}, we prove basic estimates associated to some typical tree structures. In particular, we will prove an universal deterministic estimate and a $L^p$-moment estimate for trees with cross-paired leaves. These estimates will be used systematically in the sequel.
		\vspace{0.1cm}
		\item  In Section \ref{singulartree}, we exploit the cancellation to estimate the imaginary of the cross-pairing terms $\Im \mathcal{S}_{\fk{l}',\fk{l}''}(\mathcal{T})(v)$.
		\vspace{0.1cm}
		\item  In Section \ref{nonsingular}, we estimate the remainder :
		\begin{align}\label{RT}  \mathcal{R}(\mathcal{T})(v):= \mathcal{T}_{e}(v)- \sum_{\substack{\fk{l}'\in\mathcal{L}_{\fk{r}},\fk{l}''\in\mathcal{L}_{\fk{n}_0}\\
					\iota_{\fk{l}'}+\iota_{\fk{l}''}=0 } }\mathcal{S}_{\fk{l}',\fk{l}''}(\mathcal{T})(v).
		\end{align}
		To do so, we decompose $v$ dyadically  and proceed as follows. For each dyadic piece, we first cut paired leaves within one generation. Note that this operation preserves the expanded tree structure, we will end up with a tree without paired leaves within the same generation $\mathcal{L}_{\fk{r}}, \mathcal{L}_{\fk{n}_0}$. We further proceed as:
		\begin{itemize}

			\item By definition, if two leaves of top frequencies $\fk{l}_1,\fk{l}_2$ are paired (they are in different generations), then we must have $N_{\fk{l}_3}\gtrsim N_{\fk{l}_1}^{\theta}$, the corresponding terms can be controlled using Proposition \ref{crosspaired} and we get the desired bound $p^{1-}\lambda^{2m}N_{\fk{l}_1}^{0-}$, which is dyadically summable.
			
			\item It remains to deal with the case where all $\delta$-dominated leaves are not paired (see Definition \ref{dominatedleaves}) for some small $\delta>0$ to be specified. We simply apply the Wiener chaos estimates for the first $l$-dominated frequencies together to reduce the estimate to a weighted counting bound that will be proved in Lemma \ref{lem:wcounting} in the next section.
		\end{itemize}
	\end{itemize}

	\section{Preliminaries and Weighted counting bounds}
	
In the sequel,	we denote $\langle x\rangle:=(1+|x|^2)^{1/2}$ (for $x\in\mathbb{R}^n$). 

	Given an expanded branching tree with root $\fk{r}$ and a decoration $(k_{\fk{n}})_{\fk{n}\in\mathcal{T}}$, define
	\begin{align}\label{def:orderfrequency}
		\lambda_l(\vec{k}_{\fk{r}})
		:=\mathop{\mathrm{max}}\limits_{\fk{n}:\text{children of }\fk{r}\text{ or }\fk{r}'}^{\;\;\;\;\;\;\;\;\;\;\;\;(l)}
		\bigl|k_{\fk{n}}\bigr|.
	\end{align}
	That is, $\lambda_l$ is the $l$-th largest magnitude among $\{|k_{\fk{n}}|\}$ where $\fk{n}$ ranges over the children of $\fk{r}$ its partner $\fk{r}'$.
	
	\begin{lem}\label{cancellationpsi2s}
		Let $\mathcal{T}$ be a simple tree with root $\fk{r}$ and $(k_{\fk{n}},\iota_{\fk{n}})_{\fk{n}\in\mathcal{T}}$ an admissible assignment. Then
		\begin{align}\label{boundpsi2s}
			|\psi_{2s}(\vec{k}_{\fk{r}})|
			\lesssim \lambda_1(\vec{k}_{\fk{r}})^{2s-2}\big(|\Omega(\vec{k}_{\fk{r}})|+\lambda_3(\vec{k}_{\fk{r}})^2\big).
		\end{align}
	\end{lem}
	
	Since $|\Omega(\vec{k}_{\fk{r}})|\lesssim \lambda_1(\vec{k}_{\fk{r}})\lambda_3(\vec{k}_{\fk{r}})$, it follows that
	\begin{align}\label{weight}
		|\psi_{2s}(\vec{k}_{\fk{r}})|\lesssim
		\begin{cases}
			\lambda_1(\vec{k}_{\fk{r}})^{2s-2}\big(|\Omega(\vec{k}_{\fk{r}})|+\lambda_3(\vec{k}_{\fk{r}})^2\big), & \text{if } |\Omega(\vec{k}_{\fk{r}})|\lesssim \lambda_1(\vec{k}_{\fk{r}})\lambda_3(\vec{k}_{\fk{r}}),\\[2mm]
			\lambda_1(\vec{k}_{\fk{r}})^{2s}, & \text{if } |\Omega(\vec{k}_{\fk{r}})|\sim \lambda_1(\vec{k}_{\fk{r}})^2.
		\end{cases}
	\end{align}
	
	\begin{proof}[Proof of Lemma \ref{cancellationpsi2s}]
		By admissibility, there exist $\fk{n},\fk{n}'\in\mathcal{T}$ with $|k_{\fk{n}}|\sim |k_{\fk{n}'}|\sim \lambda_1(\vec{k}_{\fk{r}})$. 
		If $\iota_{\fk{n}}+\iota_{\fk{n}'}=0$, then by the mean value theorem for $y\mapsto y^{s}$ on $\mathbb{R}_+$,
		\begin{align*}
			\big||k_{\fk{n}}|^{2s}-|k_{\fk{n}'}|^{2s}\big|
			&\sim \lambda_1(\vec{k}_{\fk{r}})^{2(s-1)}\big|\iota_{\fk{n}}|k_{\fk{n}}|^2+\iota_{\fk{n}'}|k_{\fk{n}'}|^2\big| \\
			&\le \lambda_1(\vec{k}_{\fk{r}})^{2(s-1)}\Big(|\Omega(\vec{k}_{\fk{r}})|+\Big|\sum_{\fk{m}\neq \fk{n},\fk{n}'}\iota_{\fk{m}}|k_{\fk{m}}|^2\Big|\Big)\\
			&\lesssim \lambda_1(\vec{k}_{\fk{r}})^{2(s-1)}\big(|\Omega(\vec{k}_{\fk{r}})|+\lambda_3(\vec{k}_{\fk{r}})^2\big).
		\end{align*}
		Therefore,
		\[
		|\psi_{2s}(\vec{k}_{\fk{r}})|
		\lesssim \lambda_3(\vec{k}_{\fk{r}})^{2s}
		+ \lambda_1(\vec{k}_{\fk{r}})^{2(s-1)}\big(|\Omega(\vec{k}_{\fk{r}})|+\lambda_3(\vec{k}_{\fk{r}})^2\big)
		\lesssim \lambda_1(\vec{k}_{\fk{r}})^{2(s-1)}\big(|\Omega(\vec{k}_{\fk{r}})|+\lambda_3(\vec{k}_{\fk{r}})^2\big).
		\]
		If all indices with $|k_{\fk{n}}|\sim \lambda_1(\vec{k}_{\fk{r}})$ carry the same sign, then $|\Omega(\vec{k}_{\fk{r}})|\sim \lambda_1(\vec{k}_{\fk{r}})^2$ and the trivial bound $|\psi_{2s}(\vec{k}_{\fk{r}})|\lesssim \lambda_1(\vec{k}_{\fk{r}})^{2s}$ also implies \eqref{boundpsi2s}.
	\end{proof}
	
	We use the following vector counting lemma (cf. Lemma 4.3 of \cite{DNY1}).
	
	\begin{lem}[Counting estimate]\label{threevectorcounting}
		Assume $n\ge 2$ and given dyadic numbers $L_1\ge L_2\ge\cdots\ge L_n$. 
		For $\mathbf{a}\in\mathbb{Z}^2$ and $\kappa\in\mathbb{R}$, set
		\[
		K_{L_1,\dots,L_n}(\mathbf{a},\kappa)
		:= \sum_{\substack{k_{\fk{l}_1},\dots,k_{\fk{l}_n} \\ \text{no pairing}}}
		\mathbf{1}_{\substack{\iota_{\fk{l}_1}k_{\fk{l}_1}+\cdots+\iota_{\fk{l}_n}k_{\fk{l}_n}=\mathbf{a}\\
				\iota_{\fk{l}_1}|k_{\fk{l}_1}|^2+\cdots+\iota_{\fk{l}_n}|k_{\fk{l}_n}|^2=\kappa}}
		\prod_{j=1}^n \mathbf{1}_{|k_{\fk{l}_j}|\sim L_j}.
		\]
		Then for any $\varepsilon>0$,
		\[
\sup_{(\mathbf{a},\kappa)\in\mathbb{Z}^2\times\R}		K_{L_1,\dots,L_n}(\mathbf{a},\kappa)\lesssim_{\varepsilon}
		\begin{cases}
			L_2, & n=2,\\
			L_2^{1+\varepsilon}L_3, & n=3,\\
			L_2^{1+\varepsilon}L_3\,(L_4\cdots L_n)^{2}, & n\ge 4.
		\end{cases}
		\]
	\end{lem}
	
	As a consequence of Lemma \ref{cancellationpsi2s} and Lemma \ref{threevectorcounting}, we have:
	
	\begin{cor}\label{frequentterm}
		Let $\mathcal{T}$ be a scale-$1$ expanded branching tree of size $n(\mathcal{T})=2m_0\ge4$, with root $\fk{r}$ and partner $\fk{r}'$, adapted to dyadic scales $(N_{\fk{n}})$. 
		Let $\fk{l}_1,\dots,\fk{l}_{2m_0}$ be the quasi-order of leaves so that $N_{\fk{l}_1}\ge\cdots\ge N_{\fk{l}_{2m_0}}$, and assume there is no pairing among the leaves. 
		 Then, for any $\varepsilon>0$,
		\begin{align}
			& \sum_{(k_{\fk{n}}:\fk{n}\in\mathcal{T}) } \Big|\frac{\psi_{2s}(\vec{k}_{\fk{r}})}{\langle\Omega(\vec{k}_{\fk{r}})\rangle} \Big|
			\lesssim_{\varepsilon}\ 
			N_{\fk{l}_1}^{2(s-1)}\Big(N_{\fk{l}_2}^2N_{\fk{l}_3}^2+N_{\fk{l}_3}^3N_{\fk{l}_2}^{1+\varepsilon}\Big)\prod_{j=4}^{2m_0}N_{\fk{l}_j}^2,
			\label{l1cor3.3}\\
			& \Big(\sum_{(k_{\fk{n}}:\fk{n}\in\mathcal{T})} \Big|\frac{\psi_{2s}(\vec{k}_{\fk{r}})}{\langle\Omega(\vec{k}_{\fk{r}})\rangle} \Big|^2\Big)^{\frac{1}{2}}
			\lesssim_{\varepsilon}\ 
			N_{\fk{l}_1}^{2(s-1)}\Big(N_{\fk{l}_2}N_{\fk{l}_3}+N_{\fk{l}_3}^{\frac{5}{2}}N_{\fk{l}_2}^{\frac{1}{2}+\varepsilon}\Big)\prod_{j=4}^{2m_0}N_{\fk{l}_j}.
			\label{l2cor3.3}
		\end{align}
	\end{cor}
	
	When the tree has scale $2$, we have the following (proof slightly more involved):
	
	\begin{lem}\label{lem:wcounting}
		Let $\mathcal{T}$ be a scale-$2$ expanded branching tree of size $2m_0\ge 6$, adapted to dyadic scales $(N_{\fk{n}})$. 
		Let $\fk{l}_1,\dots,\fk{l}_{2m_0}$ be the quasi-order of leaves so that $N_{\fk{l}_1}\ge\cdots\ge N_{\fk{l}_{2m_0}}$. 
		Assume there is no pairing between leaves within the same generation. Then for any $\varepsilon>0$,
		\[
		\sum_{(k_{\fk{n}}:\fk{n}\in\mathcal{T}) }
		\Big|\frac{\psi_{2s}(\vec{k}_{\fk{r}}) }{\langle\Omega(\vec{k}_{\fk{r}})\rangle } \Big|^2
		\lesssim_{\varepsilon}
		N_{\fk{l}_1}^{\,4s-2+\varepsilon}\,N_{\fk{l}_3}^{\,4}\prod_{j=4}^{2m_0}N_{\fk{l}_j}^{\,2}.
		\]
	\end{lem}
	
	\begin{proof}
		Let $\fk{n}_0$ be the unique branching node different from the root $\fk{r}$. 
		Let $\fk{l}_1',\fk{l}_2',\dots,\fk{l}_{2n_0}'$ be a quasi-order of the nodes $\mathcal{L}_{\fk{r}}\cup\{\fk{n}_0\}$ with $N_{\fk{l}_1'}\ge\cdots\ge N_{\fk{l}_{2n_0}'}$, which are precisely the leaves of the scale-$1$ tree $\mathcal{T}'$ obtained by cutting the leaves $\mathcal{L}_{\fk{n}_0}$ from $\mathcal{T}$. 
		Without loss of generality, assume $2n_0\ge 4$ (if $2n_0=2$, the weight $\psi_{2s}(\vec{k}_{\fk{r}})$ with two first-generation leaves vanishes). 
		Let $\fk{l}_1'',\dots,\fk{l}_{2\ell_0-1}''$ be a quasi-order of the leaves in $\mathcal{L}_{\fk{n}_0}$ with $N_{\fk{l}_1''}\ge\cdots\ge N_{\fk{l}_{2\ell_0-1}''}$.
		
		\smallskip
		\noindent\emph{Non-degenerate contribution:} for every $\fk{l}'\in\mathcal{L}_{\fk{r}}$, 
		\[
		\iota_{\fk{l}'}k_{\fk{l}'}+\iota_{\fk{n}_0}k_{\fk{n}_0}\neq 0.
		\]
		Set
		\[
		\mathbf{b}(k_{\fk{n}_0})
		:= \sum_{\substack{N_{\fk{l}''}\le |k_{\fk{l}''}|<2N_{\fk{l}''}\\ \fk{l}''\in\mathcal{L}_{\fk{n}_0}}}
		\mathbf{1}_{\ \iota_{\fk{n}_0}k_{\fk{n}_0}=\sum_{\fk{l}''}\iota_{\fk{l}''}k_{\fk{l}''}}.
		\]
		Then
		\[
		\sum_{(k_{\fk{n}})} \Big|\frac{\psi_{2s}(\vec{k}_{\fk{r}}) }{\langle\Omega(\vec{k}_{\fk{r}})\rangle } \Big|^2
		= \sum_{\substack{(k_{\fk{l}'})\\ N_{\fk{l}'}\le |k_{\fk{l}'}|<2N_{\fk{l}'}}}
		\Big|\frac{\psi_{2s}(\vec{k}_{\fk{r}})}{\langle\Omega(\vec{k}_{\fk{r}})\rangle}\Big|^2\cdot \mathbf{b}(k_{\fk{n}_0}),
		\]
		where $(k_{\fk{l}'})$ denotes a decoration of the scale-$1$ tree $\mathcal{T}'$.
		Using Lemma \ref{cancellationpsi2s}, 
		\[
		\Big|\frac{\psi_{2s}(\vec{k}_{\fk{r}})}{\langle\Omega(\vec{k}_{\fk{r}})\rangle}\Big|^2
		\lesssim N_{\fk{l}_1'}^{\,2(2s-2)}\Big(1+\frac{N_{\fk{l}_3'}^{\,4}}{\langle\Omega(\vec{k}_{\fk{r}})\rangle^2}\Big),
		\]
		and applying Lemma \ref{threevectorcounting} (the $n=3$ case for the constraint involving $k_{\fk{l}_1'},k_{\fk{l}_2'},k_{\fk{l}_3'}$), we obtain
		\begin{align*}
			\sum_{\substack{(k_{\fk{l}'}:\fk{l}'\in\mathcal{T}')\\ \text{non-degenerate} } }	\Big|\frac{\psi_{2s}(\vec{k}_{\fk{r}}) }{\langle\Omega(\vec{k}_{\fk{r}})\rangle } \Big|^2\cdot \mathbf{b}(k_{\fk{n}_0})\lesssim_{\epsilon} &
			\sup_{k_{\fk{n}_0}}|\mathbf{b}(k_{\fk{n}_0})|\cdot
			N_{\fk{l}_1'}^{4s-4}(N_{\fk{l}_2'}^{1+\epsilon} N_{\fk{l}_3'}^5+N_{\fk{l}_2'}^2N_{\fk{l}_3'}^2)\prod_{j=4}^{2n_0}N_{\fk{l}_j'}^2\\
			\lesssim &
			N_{\fk{l}_1'}^{4s-2+\epsilon}N_{\fk{l}_3'}^4\Big(\prod_{j=4}^{2n_0}N_{\fk{l}_j'}^2\Big)\cdot\Big(\prod_{j=2}^{2\ell_0-1}N_{\fk{l}_j''}^2\Big)\\
			\sim & N_{\fk{l}_1'}^{4(s-1)+\epsilon}N_{\fk{l}_3'}^2\cdot \prod_{j=2}^{2n_0}N_{\fk{l}_j'}^2\cdot\prod_{j=2}^{2\ell_0-1}N_{\fk{l}_j''}^2.
		\end{align*}
		A moment of thinking leads to the obvious bound $N_{\fk{l}_1'}^{4(s-1)+\epsilon}N_{\fk{l}_3'}^2\lesssim N_{\fk{l}_1}^{4(s-1)+\epsilon}N_{\fk{l}_3}^2$. 
		To conclude, we need to verify that
		$$
		\Big(\prod_{j=2}^{2n_0}N_{\fk{l}_j'}^2\Big)\cdot\Big(\prod_{j=2}^{2\ell_0-1}N_{\fk{l}_j''}^2\Big)
		\lesssim \prod_{j=2}^{2m_0}N_{\fk{l}_j}^2.$$
		Indeed, we note $N_{\fk{l}_1'}\sim N_{\fk{l}_2'}$ and $N_{\fk{l}_1}\sim N_{\fk{l}_2}$. If $\fk{n}_0\in\{\fk{l}_1',\fk{l}_2'\}$, then from $N_{\fk{l}_1'}\sim N_{\fk{n}_0}\lesssim N_{\fk{l}_1''}$, we have $N_{\fk{l}_1''}\sim N_{\fk{l}_1}$. Hence the left hand side of the desired inequality is comparable to $$\frac{\prod_{\fk{l}\in\mathcal{L}}N_{\fk{l}}^2}{N_{\fk{l}_1''}^2}\sim \prod_{j=2}^{2m_0}N_{\fk{l}_j}^2.
		$$
		If $\fk{n}_0\notin\{\fk{l}_1',\fk{l}_2'\}$, the left hand side equals to
		$$ \prod_{\fk{l}\in\mathcal{L}}N_{\fk{l}}^2\cdot \frac{ N_{\fk{n}_0}^2 }{\max_{\mathcal{L}_{\fk{r}}}N_{\fk{l}'}^2\cdot\max_{\mathcal{L}_{\fk{n}_0}}N_{\fk{l}''}^2 }\lesssim \prod_{j=2}^{2m_0}N_{\fk{l}_j}^2,
		$$
		where we use the fact that $N_{\fk{n}_0}\lesssim \min\{
		\max_{\mathcal{L}_{\fk{r}}}N_{\fk{l}'},\max_{\mathcal{L}_{\fk{n}_0}}N_{\fk{l}''}
		\}$.
		
		\smallskip
		\noindent\emph{Degenerate contribution:}	Suppose that 
		$$ \iota_{\fk{n}_0}k_{\fk{n}_0}+\iota_{\fk{l}_{j_0}'}k_{\fk{l}_{j_0}'}=0,\quad \text {for some } \fk{l}_{j_0}'\in\mathcal{L}_{\fk{r}}.
		$$
		In this case, the weight in the sum $\Big|\frac{\psi_{2s}(\vec{k}_{\fk{r}})}{\langle\Omega(\vec{k}_{\fk{r}})\rangle} \Big|^2$ degenerate to a similar weight depending only on $k_{\fk{l}'}, \fk{l}'\in\mathcal{L}_{\fk{r}}\setminus\{\fk{l}_{j_0}'\}$. Consequently, the sum $\sum_{(k_{\fk{n}}: \fk{n}\in\mathcal{T}) }
		\Big|\frac{\psi_{2s}(\vec{k}_{\fk{r}}) }{\langle\Omega(\vec{k}_{\fk{r}})\rangle } \Big|^2$
		breaks into the product of two sums corresponding to two disjoint scale-1 expanded branching trees $\mathcal{T}_1,\mathcal{T}_2$, where $\mathcal{T}_1$ is obtained by cutting the nodes $\fk{n}_0,\fk{l}_{j_0}'$ and leaves $\mathcal{L}_{\fk{n}_0}$, while the tree $\mathcal{T}_2$ is obtained by cutting all the leaves in $\mathcal{L}_{\fk{r}}$ expect $\fk{l}_{j_0}'$, making the $\fk{n}_0$ the new root and $\fk{l}_{j_0}'$ the partner leave. Expressed in formula,
		\begin{align}\label{eq:splitting}
			\sum_{(k_{\fk{n}}: \fk{n}\in\mathcal{T}) }
			\Big|\frac{\psi_{2s}(\vec{k}_{\fk{r}}) }{\langle\Omega(\vec{k}_{\fk{r}})\rangle } \Big|^2\mathbf{1}_{\iota_{n_0}k_{\fk{n_0}}+\iota_{\fk{l}_{j_0}'}k_{\fk{l}_{j_0}'}=0 }=
			\Big(\sum_{(k_{\fk{n}}:\fk{n}\in\mathcal{T}_1)  }
			|\frac{\psi_{2s}(\vec{k}_{\fk{r}}) }{\langle\Omega(\vec{k}_{\fk{r}})\rangle } \Big|^2
			\Big)\cdot\Big(\sum_{(k_{\fk{n}}:\fk{n}\in\mathcal{T}_2) } 1 \Big).
		\end{align}
		Abusing the notation, we denote $\fk{l}_1',\fk{l}_2',\cdots,\fk{l}_{2n_0-2}'$ a quasi-order of leaves in the tree $\mathcal{T}_1$ while still keeping the the notation of quasi-order $\fk{l}_1'',\cdots,\fk{l}_{2\ell_0-1}''$ of leaves in $\mathcal{L}_{\fk{n}_0}$, the partner leave of $\mathcal{T}_2$, paired with the root $\fk{n}_0$ will be denoted as $\fk{n}_0'$. Using Corollary \ref{frequentterm} \footnote{Strictly speaking, we should assume that there is no other pairing within the leaves of $\mathcal{T}_1$. We ignore this issue, since we can always eliminate paired leave to reduce the size of the scale-1 tree. }, the right hand side of \eqref{eq:splitting} can be controlled crudely by (we ignore the zero contribution if $2n_0=4$)
		\begin{align*}
			&N_{\fk{l}_1'}^{4s-4+\epsilon}N_{\fk{l}_3'}^2\cdot\Big(\prod_{j=2}^{2n_0-2}N_{\fk{l}_j'}^2\Big) \cdot\frac{\prod_{j=1}^{2\ell_0-1}N_{\fk{l}_j''}^2}{\max\{N_{\fk{n}_0'}^2,N_{\fk{l}_1''}^2 \} } \\
			\lesssim & N_{\fk{l}_1'}^{4s-4+\epsilon}N_{\fk{l}_3'}^2\cdot\prod_{j=2}^{2m_0}N_{\fk{l}_j}^2\lesssim N_{\fk{l}_1}^{4s-2+\epsilon}N_{\fk{l}_3}^4\prod_{j=4}^{2m_0}N_{\fk{l}_j}^2.
		\end{align*}
		The proof of Lemma \ref{lem:wcounting} is now complete.
	\end{proof}

	Let us recall in this section the Wiener chaos estimate for multi-linear expression of complex Gaussian random variables. In the sequel we adopt the notation $z^{+}=z$ and $z^-=\ov{z}$ for a complex number $z\in\C$. From \cite{DNY1}, we have : 
	\begin{lem}[Wiener chaos estimate]\label{pairingWiener} 
		Consider the multi-linear expression of Gaussian:
		$$ F(\omega',\omega)=\sum_{k_{1},\cdots,k_{n}}c_{k_{1},\cdots,k_{n}}(\omega')\cdot \prod_{j=1}^{n}g_{k_{j} }  ^{\iota_{j}}(\omega),
		$$
		where the random variables $c_{k_{1},\cdots,k_{n} }(\omega')$ are independent of $g_{k_{j } }  ^{\iota_{j}}(\omega) $. Then for any $p\geq 2$, we have
		$$ \|F(\omega',\omega)\|_{L_{\omega}^p}\leq Cp^{\frac{n}{2}}\Big(\sum_{(X,Y)}\sum_{(k_{i}):i\notin X\cup Y }\Big(\sum_{\text{pairing }(k_{i_{\nu}},k_{j_{\nu}} )  }
		|c_{k_{1},\cdots, k_{n}}| \Big)^2\Big)^{\frac{1}{2}},
		$$
		where in the above summation, we require that $X=\{i_1,\cdots,i_{l}\}$, $Y=\{j_1,\cdots,j_{l}\}$ are two disjoint subsets of $\{1,\cdots,n\}$. Moreover, the last paring summation means
		$$ \sum_{\text{paring} (k_{i_{\nu}},k_{j_{\nu}}) }(\cdots)=\sum_{\sigma\in \fk{S}(X,Y)}\sum_{\text{paring }(k_i,k_{\sigma(i)}):i\in X }(\cdots),
		$$
		where $\fk{S}(X,Y)$ is the set of bijections from $X$ to $Y$.
		
	\end{lem}


	\section{Estimates for basic configurations}\label{Sec:Basicestimates}

	\subsection{Deterministic basic tree estimates}

	Recall the energy multilinear form defined in \eqref{Te}. We first establish the following deterministic basic tree estimates. Recall that the sequential Sobolev space $h^s$ is defined via the norm
	$$ \|a\|_{h^s}:=\|\langle k\rangle^s a_k \|_{l^2}.
	$$
	\begin{prop}[Basic tree estimates]\label{Basictreeestimate}
		Let $n_0\geq 2$ and $a^{(1)},\cdots, a^{(2n_0)}\in h^{s}(\Z^2)$. \\
		$\mathrm{(1)}$ For an expanded branching tree $\mathcal{T}$ of scale $s(\mathcal{T})=1$ and size $n(\mathcal{T})=2n_0$, we have
		\begin{align}\label{eq:generaltree} 
			&|\mathcal{T}_e(a^{(1)},\cdots,a^{(2n_0)})|\notag \\ \lesssim 
			&\mathrm{max}_j^{(1)}\|a^{(j)}\|_{h^{(s-1)+}}
			\mathrm{max}_j^{(2)}\|a^{(j)}\|_{h^{s-1}}
			\mathrm{max}_j^{(3)}\|a^{(j)}\|_{h^{2+}}
			\mathrm{max}_j^{(4)}\|a^{(j)}\|_{l^{2}}
			\prod_{l=5}^{2n_0}
			\mathrm{max}_j^{(l)}\|a^{(j)}\|_{h^{1+}}.
		\end{align}
		In particular, if the tree $\mathcal{T}$ is adapted to dyadic scales $(N_1,\cdots,N_{2n_0})$ with $N_{(1)}\geq N_{(2)}\geq\cdots N_{(2n_0)}$ the non-increasing rearrangement of $N_1,\cdots, N_{2n_0}$, then
		\begin{align}\label{eq:dyadictreeestimate} 
			|\mathcal{T}_e(a^{(1)},\cdots, a^{(2n_0)})|\lesssim N_{(1)}^{2s-2+}N_{(3)}^{2+}\prod_{j=5}^{2n_0}N_{(j)}\prod_{j=1}^{2n_0}\|a^{(j)}\|_{l^2}.
		\end{align}
		$\mathrm{(2)}$ For an expanded branching tree $\mathcal{T}$ of scale $s(\mathcal{T})=2$ and size $n(\mathcal{T})=2n_0$, we have
		\begin{align}\label{eq:generaltree'} 
			&|\mathcal{T}_e(a^{(1)},\cdots,a^{(2n_0)})|\notag \\ \lesssim 
			&\mathrm{max}_j^{(1)}\|a^{(j)}\|_{h^{(s-1)+}}
			\mathrm{max}_j^{(2)}\|a^{(j)}\|_{h^{s-1}}
			\mathrm{max}_j^{(3)}\|a^{(j)}\|_{h^{2+}}
			\mathrm{max}_j^{(2n_0)}\|a^{(j)}\|_{l^{2}}
			\prod_{\substack{4\leq l\leq 2n_0-1
					}}
			\mathrm{max}_j^{(l)}\|a^{(j)}\|_{h^{1+}}.
		\end{align}
		In particular, if the tree $\mathcal{T}$ is adapted to dyadic scales $(N_1,\cdots,N_{2n_0})$ with $N_{(1)}\geq N_{(2)}\geq\cdots N_{(2n_0)}$ the non-increasing rearrangement of $N_1,\cdots, N_{2n_0}$, then
		\begin{align}\label{eq:dyadictreeestimate'} 
			|\mathcal{T}_e(a^{(1)},\cdots, a^{(2n_0)})|\lesssim N_{(1)}^{2s-2+}N_{(3)}^{2+}N_{(2n_0)}^{-1}\prod_{j=4}^{2n_0}N_{(j)}\prod_{j=1}^{2n_0}\|a^{(j)}\|_{l^2}.
		\end{align}
	\end{prop}

	\begin{rem}
		For scale-2 trees the estimate in \eqref{eq:dyadictreeestimate'} loses a factor of $N_{(4)}/N_{(2n_0)}$ compared with the scale-1 bound, but this is a moderate loss compared to the dominant one $N_{(3)}^{2+}$.
	\end{rem}

	\begin{proof}
		(1)		Let $\mathcal{T}$ be such that $s(\mathcal{T})=1$. By decomposing $a^{(j)}$ dyadically,  it suffices to prove \eqref{eq:dyadictreeestimate}.
		Therefore, without loss of generality, we assume that each function $a^{(j)}$ is already localized dyadically as $a^{(j)}=\mathbf{1}_{N_j/2<|k_j|\le 2N_j}a^{(j)}$. 
		
		By relabeling the leaves as the index set $I=\{1,2,\cdots,2n_0\}$ and using Lemma \ref{cancellationpsi2s}, we have
		\begin{align*}
			|\mathcal{T}_e(a^{(1)},\cdots, a^{(2n_0)})|\lesssim \sum_{1\leq|\kappa|\lesssim N_{(1)}N_{(3)} }\frac{N_{(1)}^{2s-2}(N_{(3)}^2+|\kappa|) }{|\kappa|}
			\sum_{\substack{k_1-k_2+\cdots-k_{2n_0}=0\\
					\Omega(\vec{k}_{\fk{r}})=\kappa  } }\prod_{j=1}^{2n_0}|a_{k_j}^{(j)}|. 
		\end{align*}
		First, we bound the contribution
		\begin{align*}
			\mathrm{I}:=\sum_{1\leq |\kappa|\lesssim N_{(1)}N_{(3)} } N_{(1)}^{2s-2}
			\sum_{\substack{k_1-k_2+\cdots-k_{2n_0}=0\\
					\Omega(\vec{k}_{\fk{r}})=\kappa  } }\prod_{j=1}^{2n_0}|a_{k_j}^{(j)}|
		\end{align*}
		as (to fix the idea, we assume that $|k_1|\sim |k_2|\sim N_{(1)}$)
		\begin{align*}
			& N_{(1)}^{2s-2}\sum_{k_1,k_3,\cdots,k_{2n_0}}|a_{k_1}^{(1)}||a_{k_1+k_3-k_4+\cdots-k_{2n_0}}^{(2)}|\prod_{j=3}^{2n_0}|a_{k_j}^{(j)}|\leq N_{(1)}^{2s-2}\|a^{(1)}\|_{l^2}\|a^{(2)}\|_{l^2}\prod_{j=3}^{2n_0}\|a^{(j)}\|_{l^1}\\
			\lesssim & N_{(1)}^{2s-2}\|a^{(1)}\|_{l^2}\|a^{(2)}\|_{l^2}\prod_{j=3}^{2n_0}\|a^{(j)}\|_{h^{1+}}\lesssim N_{(1)}^{2s-2}\prod_{j=3}^{2n_0}N_{(j)}\prod_{j=1}^{2n_0}\|a^{(j)}\|_{l^2},
		\end{align*}
		where we used the Cauchy-Schwarz inequality $\|b_k\|_{l^1(\mathbb{Z}^2)}\lesssim N\|b_k\|_{l^{2}(\Z^2)}$, provided that $b_k=b_k\mathbf{1}_{|k|\sim N}$.
		
		Next, we control the contribution
		$$ \mathrm{II}:=\sum_{1\leq|\kappa|\lesssim N_{(1)}N_{(3)} }\frac{N_{(1)}^{2s-2}N_{(3)}^2 }{|\kappa|}
		\sum_{\substack{k_1-k_2+\cdots-k_{2n_0}=0\\
				\Omega(\vec{k}_{\fk{r}})=\kappa  } }\prod_{j=1}^{2n_0}|a_{k_j}^{(j)}|
		$$
		using the Strichartz inequality.
		Set 
		$$ A^{(j)}(x):=\sum_{k_j}|a_{k_j}^{(j)}|\e^{ik_j\cdot x},
		$$
		Then for any permutation  $\tau\in\fk{S}(I)$ 
		$$ \sum_{\substack{ k_1-k_2+\cdots-k_{2n_0}=0\\
				\Omega(\vec{k}_{I})=\kappa } } \prod_{j=1}^{2n_0}|a_{k_j}^{(j)}|=\int_0^{2\pi}\int_{\T^2}\e^{it\kappa}\prod_{j=1}^{4}(\e^{it\Delta}A^{\tau(j)})^{\iota_{\tau(j)}}\cdot \prod_{j=5}^{2n_0}(\e^{it\Delta}A^{\tau(j)})^{\iota_{\tau(j)}}dxdt.
		$$
		By the bilinear refinement of the $L^4$-Strichartz inequality (Proposition 3.5 of \cite{BGT}) on $\T^2$, the absolute value of the above quantity is bounded by
		\begin{align*}
			\mathrm{max}_{j=1,2,3,4}^{(3)}\{N_{\tau(j)}^{0+}\}\Big(\prod_{j=1}^4\|A^{\tau(j)}\|_{l^2}\Big)\cdot \prod_{j=5}^{2n_0}\|e^{it\Delta}A^{\tau(j)}\|_{L_{t,x}^{\infty}}.
		\end{align*}
		Hence by Bernstein, we obtain that
		$$\sum_{\substack{ k_1-k_2+\cdots-k_{2n_0}=0\\
				\Omega(\vec{k}_{I})=\kappa } } \prod_{j=1}^{2n_0}|a_{k_j}^{(j)}|\lesssim \min_{\tau\in\fk{S}(I)}\|a^{\tau(1)}\|_{h^{0+}}\Big(\prod_{j=2}^4\|a^{\tau(j)}\|_{l^2}\Big)\cdot\Big(\prod_{j=5}^{2n_0}\|a^{\tau(j)}\|_{h^{1}} \Big) 
		$$
		Therefore,
		$$ \mathrm{II}\lesssim N_{(1)}^{2s-2+}N_{(3)}^{2+}\Big(\prod_{j=5}^{2n_0}N_{(j)}\Big)\prod_{j=1}^{2n_0}\|a^{(j)}\|_{l^2}.
		$$
		\vspace{0.3cm}

		(2)	Next we study the situation where $s(\mathcal{T})=2$. Our argument is to reduce the estimate to (1). We label the children of the non-root node $\fk{n}_0$ as $1,2,\cdots, 2m_0-1$ for some $m_0< n_0$. As before, we assume that each $a^{(j)}$ is supported in $N_j/2\leq |k_j|\leq 2N_j$. 
		Denote by $\widetilde{\mathcal{T}}$ the subtree with root $\fk{n}_0$ and $\mathcal{T}'$ the expanded branching tree of scale $1$ obtained by deleting all leaves of the simple tree $\widetilde{\mathcal{T}}$. Let
		$$ \widetilde{a}^{(1)}=\prod_{j=1}^{2m_0-1}a^{(j)}, \; \widetilde{a}^{(2)}=a^{(2m_0)},\cdots, \widetilde{a}^{2l_0}=a^{2n_0},
		$$
		where $l_0=n_0-m_0+1$.
		Thus
		$$ \mathcal{T}_e(a^{(1)},\cdots, \cdots, a^{(2n_0)})=\mathcal{T}_e'(\widetilde{a}^{(1)},\cdots, a^{(2l_0)}).
		$$
		In order to apply (1), we need to estimate the $L^2$ norm of $\widetilde{a}^{(1)}$.
		
		Denote by $P_{(1)}\geq P_{(2)}\geq\cdots P_{(2m_0-1)}$ the non-increasing rearrangement of frequency scales of leaves of $\widetilde{\mathcal{T}}$ $N_1,\cdots, N_{2m_0-1}$, by Cauchy-Schwarz, we have
		\begin{align}\label{T0estimate} 
			\|\widetilde{a}^{(1)}\|_{L^2}^2\leq  &\sum_{k}\Big(\sum_{\substack{k_1-k_2+\cdots+k_{2m_0-1}=k \\
					|k_j|\sim N_j	 
			} } \prod_{j=1}^{2m_0-1}|a_{k_j}^{(j)}| \Big)^2 \notag  \\
			\leq &\sum_{k_1,\cdots,k_{2m_0-1} }\prod_{j=1}^{2m_0-1}|a_{k_j}^{(j)}|^2\cdot \sup_{k}\sum_{\substack{k_1-k_2+\cdots+k_{2m_0-1}=k\\ |k_j|\sim N_j  } }1\notag  \\
			\lesssim &\Big(\prod_{j=2}^{2m_0-1}P_{(j)}^2\Big)\prod_{j=1}^{2m_0-1}\|a^{(j)}\|_{l^2}^2.
		\end{align} 
		Together with the estimate for scale-1 tree, we obtain that
		\begin{align*}
			| \mathcal{T}_e(a^{(1)},\cdots, \cdots, a^{(2n_0)}) |\lesssim
			M_{(1)}^{2s-2+}M_{(3)}^{2+}\Big(\prod_{j=5}^{2l_0}M_{(j)}\Big)
			\Big(\prod_{l=2}^{2m_0-1} P_{(l)}
			\Big) \prod_{j=1}^{2n_0}\|a^{(j)}\|_{l^2},
		\end{align*}
		where $M_{(1)}\geq \cdots M_{(2l_0)}$ be the non-increasing rearrangement of dyadic scales $N_{\fk{n}_0}, N_{2m_0}, N_{2m_0+1},\cdots, N_{2n_0}$ for leaves of $\mathcal{T}'$. 
		A moment check leads to 
		$$ 	M_{(1)}^{2s-2+}M_{(3)}^{2+}\Big(\prod_{j=5}^{2l_0}M_{(j)}\Big)
		\Big(\prod_{l=2}^{2m_0-1} P_{(l)}
		\Big)\lesssim N_{(1)}^{2s-2+}N_{(3)}^{2+}N_{(2n_0)}^{-1}\prod_{j=4}^{2n_0}N_{(j)}.
		$$
		We complete the proof of Proposition \ref{Basictreeestimate}.
	\end{proof}
	
	\subsection{$L^p$-moment bounds for the cross-paired tree estimates}\label{sub:cross-paired}
	
	We fix an expanded branching tree $\mathcal T$ of scale 2, adapted to dyadic sizes $(N_{\fk n})_{\fk n\in\mathcal T}$. The most technical step is to control pairings between high‐frequency leaves lying in different generations.
	To do so, we relabel the children leaves of $\fk{r}$ as $\fk{r}_1,\cdots,\fk{r}_{2m_0-1}$ and $\fk{r}'=\fk{r}_{2m_0}$. Without loss of generality, we assume that the unique non-root node $\fk{n}=\fk{r}_1$ and denote its children $\fk{n}_1,\cdots,\fk{n}_{2l_0-1}$. Note that the total number of leaves is $2n_0=2m_0-1+2l_0-1$.

	In order to distinguish the paired leaves with high frequencies in a given decoration, we introduce the notion of dominated leaves.  Recall the notion of quasi-order for leaves in a decoration adapted to dyadic numbers $(N_{\fk{n}}:\fk{n}\in\mathcal{T})$.  
	\begin{definition}
		Let $\mathcal{T}$ be an expanded branching tree adapted to dyadic numbers $(N_{\fk{n}}:\fk{n}\in\mathcal{T})$. For $0<\delta<1$, define
		\begin{align}\label{dominatedleaves} \mathcal{L}_{\delta}^*:=\{\fk{l}\in\mathcal{L}: N_{\fk{l}}\geq N_{\fk{l}_1}^{1-\delta}\}
		\end{align}
		be the set of \emph{$\delta$-dominated leaves}, where we recall that $(\fk l_1,\dots,\fk l_{2n_0})$ is the quasi‐ordered list of all leaves.
	\end{definition}

	
	\begin{definition}
		A pairing $\mathcal P=(X,Y,\sigma)$ on $\mathcal L$ is called \emph{$\delta$-dominating} if 
		\[
		X,Y\subset\mathcal L_\delta^*.
		\]
		We abbreviate such a triple by $(X,Y,\sigma;\,\mathcal L_\delta^*)$.
	\end{definition}
	
	\begin{definition}
		A decoration $(k_{\fk n})_{\fk n\in\mathcal T}$ is said to be \emph{adapted} to the $\delta$-dominating pairing $(X,Y,\sigma;\mathcal L_\delta^*)$ if :
		\begin{itemize}
			\item for every $\fk l\in X$, one has 
			$
			k_{\fk l}=k_{\sigma(\fk l)}, \iota_{\fk l}+\iota_{\sigma(\fk l)}=0,
			$
			\item for all unpaired $\fk l,\fk l'\in\mathcal L_\delta^*\setminus (X\cup Y)$,
			$
			k_{\fk l}\neq k_{\fk l'}$, or
			$	\iota_{\fk l}=\iota_{\fk l'}.
			$
		\end{itemize}
	\end{definition}
	
	In other words, in a decoration adapted to a $\delta$-dominating pairing,  dominated leaves in $X\cup Y$ are paired according to $\sigma$, and there is no another paring in $\delta$-dominated leaves $\mathcal{L}_{\delta}^*$.
	
	According to the definitions above, we  decompose $\mathcal{T}_e(a^{(1)},\cdots,a^{(2n_0)})$ according to given assignments of pairings:
	$$ \sum_{(X,Y,\sigma;\mathcal{L}_{\delta}^*)}
	\mathcal{T}_{e,X,Y}^{\sigma}(a^{(1)},\cdots,a^{(2n_0)} ),
	$$
	where
	\begin{align}\label{crosspairedtree}
		\mathcal{T}_{e,X,Y}^{\sigma}(a^{(1)},\cdots,a^{(2n_0)} ):=\sum_{ (k_{\fk{n}}:\fk{n}\in\mathcal{T})_{(X,Y,\sigma;\mathcal{L}_{\delta}^*)} }\frac{\psi_{2s}(\vec{k}_{\fk{r}})}{\Omega(\vec{k}_{\fk{r}})}\prod_{j=1}^{2n_0}(a_{k_{\fk{l}_j}}^{(\fk{l}_j)})^{\iota_{\fk{l}_j}}.
	\end{align}

	In the rest of this section, we prove the following proposition of the $L^p$-moment bounds for trees with at least two dominated paired leaves  between different generations : 
	\begin{prop}[Cross-paired tree estimate]
		\label{crosspaired}
		Let $0<\delta<\tfrac{1}{100(n_0+s)}$ and let $\mathcal T$ be a scale-2 expanded tree of size $2n_0$, quasi-ordered so that 
		\[
		N_{\fk l_1}\ge N_{\fk l_2}\ge\cdots\ge N_{\fk l_{2n_0}}.
		\]
		Write $\fk r$ for the root, $\fk r_1$ for the other (unique) branching node, and set 
		$
		l=\#\mathcal L_\delta^*\;\ge2.
		$
		Fix a $\delta$-dominating pairing $(X,Y,\sigma)$ with 
		$
		X\subset\mathcal L_{\fk{r}_1}, 
		Y\subset\mathcal L_{\fk r}.
		$
		Then for each projection $a^{(j)}=\mathbf P_{N_{\fk l_j}}\phi$, for all $p\ge2$ and any small $\epsilon>0$, any $\lambda\geq 1$, one has
		\[
		\bigl\|\mathcal T_{e,X,Y}^\sigma\,\mathbf1_{\|\phi\|_{H^1}\le\lambda}\bigr\|_{L^p_{d\mu_s}}
		\;\lesssim_\epsilon\;
		p^{1-\frac\epsilon{2s}}\,
		\lambda^{2n_0}\,
		N_{\fk l_1}^{\epsilon}
		\Bigl[
		N_{\fk l_1}^{-2(s-1)\delta\,(1-\frac\epsilon{2s})}
		+N_{\fk l_3}^{-2\,(1-\frac\epsilon{2s})}
		\Bigr],
		\]
		when $l=2$, and
		\[
		\bigl\|\mathcal T_{e,X,Y}^\sigma\,\mathbf1_{\|\phi\|_{H^1}\le\lambda}\bigr\|_{L^p_{d\mu_s}}
		\;\lesssim_\epsilon\;
		p^{1-\epsilon}\,
		\lambda^{2n_0}\,
		N_{\fk l_1}^{-\frac{l-2}{l}},
		\]
		for $l\ge3$.  The implicit constants are independent of $p$.
	\end{prop}
	
	\begin{proof}
		
		We organize the proof in three steps.
		As the first simplification, we may assume that in the analysis below, no inner pairings within the same generation for dominated frequencies.

		\noi
		$\bullet${\bf Step 1 : Deterministic estimate} :
		We begin by isolating the two paired leaves and the next two largest unpaired leaves:
		
		\begin{itemize}
			\item Let $i_0<j_0$ be the indices of the first dominated leaf in the second generation and its partner in the root generation:
			\begin{align}\label{i0j0}
				i_0=\min\{i:\fk l_i\in X\}, 
				\quad
				j_0=\sigma(i_0).
			\end{align}
			\item Let $i_1<j_1$ be the two largest indices among the remaining leaves:
			\begin{align}\label{i1j1}
				i_1=\min\{i\notin\{i_0,j_0\}\}, 
				\quad
				j_1=\min\{j\notin\{i_0,j_0,i_1\}\}.
			\end{align}
		\end{itemize}
		
		By construction, each paired leaf satisfies
		\(\,N_{\fk l_{i_0}}=N_{\fk l_{j_0}}\ge N_{\fk l}\) for all \(\fk{l}\in X\cup Y\).  Moreover one checks that
		\[
		N_{\fk{l}_1}^2\,N_{\fk{l}_3}^2
		\;\prod_{r\in\{i_0,j_0,i_1,j_1\}}N_{\fk l_r}^{-1}
		\;\le\;
		1 \;+\; N_{\fk{l}_3}^2\,N_{\fk{l}_1}^{-2(1-\delta)}.
		\]
		Using this and Lemma \ref{cancellationpsi2s}, we have
		\begin{align}\label{I} 
			|\mathcal{T}_{e,X,Y}^{\sigma}|\lesssim &\sum_{k_{\fk{l}_{i_0}},k_{\fk{l}_{j}},j\neq i_0,j_0  } \mathrm{1}_{\sum_{ j}\iota_{\fk{l}_j}k_{\fk{l}_j}=0 }
			\frac{N_{\fk{l}_1}^{2s-2}(N_{\fk{l}_3}^2+|\Omega(\vec{k}_{\fk{r}})|) }{|\Omega(\vec{k}_{\fk{r}})|}
			|a_{k_{\fk{l}_{i_0}}}^{(\fk{l}_{i_0})} 
			a_{k_{\fk{l}_{i_0}}}^{(\fk{l}_{j_0})}|\prod_{j\neq i_0,j_0}^{}|a_{k_{\fk{l}_j}}^{(\fk{l}_j)}| \notag \\
			\leq & 2N_{\fk{l}_1}^{2s-2}N_{\fk{l}_3}^2\|a^{(\fk{l}_{i_0})}\|_{l^2}\|a^{(\fk{l}_{j_0})} \|_{l^2} \sum_{k_{\fk{l}_j}:j\neq i_0,j_0 } \mathbf{1}_{\sum_{j\neq i_0,j_0 }\iota_{\fk{l}_j}k_{\fk{l}_j}=0}\prod_{j\neq i_0,j_0}|a_{k_{\fk{l}_j}}^{(\fk{l}_j)}| \notag \\
			\lesssim &N_{\fk{l}_1}^{2s-2}N_{\fk{l}_3}^2\Big(\prod_{j= i_0,j_0,i_1,j_1}\|a_{k_{\fk{l}_j}}^{(\fk{l}_j)}\|_{l^2}\Big)\Big(\prod_{j\neq  i_0,j_0,i_1,j_1}\|a_{k_{\fk{l}_j}}^{(\fk{l}_j)}\|_{l^1}
			\Big) \notag \\
			\lesssim &N_{\fk{l}_1}^{2s-4}\cdot N_{\fk{l}_1}^2N_{\fk{l}_3}^2\Big(\prod_{j= i_0,j_0,i_1,j_1}\|a_{k_{\fk{l}_j}}^{(\fk{l}_j)}\|_{l^2}\Big)\Big(\prod_{j\neq i_0,j_0,i_1,j_1}\|a_{k_{\fk{l}_j}}^{(\fk{l}_j)}\|_{h^1}
			\Big)
			\Big)\notag \\
			\lesssim &N_{\fk{l}_1}^{2(s-2)+2\delta}\prod_{j=1}^{2n_0}\|a^{(\fk{l}_j)}\|_{h^1}\leq N_{\fk{l}_1}^{2(s-2)+2\delta}\lambda^{2n_0}.
		\end{align}
		
		\noi
		$\bullet${\bf Step 2: $L^p$-moment estimate using Wiener chaos :}
		Recall that $a^{(\fk{l}_j)}=\mathbf{P}_{N_{\fk{l}_j}}\phi$, we write
		\begin{align*}
			\mathcal{T}_{e,X,Y}^{\sigma}=\sum_{\substack{(k_{\fk{n}}: \fk{n}\in\mathcal{T} )\\ 
					\iota_{\fk{l}}k_{\fk{l}}+\iota_{\sigma(\fk{l})}k_{\sigma(\fk{l})}=0,\fk{l}\in X
			} } \frac{\psi_{2s}(\vec{k}_{\fk{r}})}{\Omega(\vec{k}_{\fk{r}})}\Big(\prod_{j=1}^{l}\frac{g_{k_{\fk{l}_j}}^{\iota_{\fk{l}_j}} 
			}{\langle k_{\fk{l}_j}\rangle^s }\Big)\cdot\Big(\prod_{j=l+1}^{2n_0}(a_{k_{\fk{l}_j}}^{(\fk{l}_j)})^{\iota_{\fk{l}_j}} \Big).
		\end{align*}
		Since functions $a_{k_{\fk{l}_j}}^{(\fk{l}_j)}, j> l$ are independent of $g_{k_{\fk{l}_j}},j\leq l$, using the Wiener chaos for the first $l$ Gaussians (Lemma \ref{pairingWiener}), we have
		\begin{align}\label{LpI}
			&\|\mathcal{T}_{e,X,Y}^{\sigma}\cdot \mathbf{1}_{\|\phi\|_{H^1}\leq \lambda}\|_{L_{d\mu_s}^p}\notag \\ \lesssim &p^{\frac{l}{2} }
			\Big(\sum_{\substack{ k_{\fk{l}_i}: 1\leq i\leq l
					\\ \fk{l}_i\notin X\cup Y } } \Big(\sum_{
				\substack{k_{\fk{l}},\fk{l}\in X, \\
					\iota_{\fk{l}}k_{\fk{l}}+\iota_{\sigma(\fk{l})}k_{\sigma(\fk{l})}=0,\fk{l}\in X
				} 
			} |\Upsilon(k_{\fk{l}_1},k_{\fk{l}_2}\cdots,k_{\fk{l}_{l}} )|
			\Big)^2   \Big)^{\frac{1}{2}},
		\end{align}
		where\footnote{Here and in the sequel, we always replace $|g_k|^2$ by $1$. } 
		\begin{align}\label{Upsilon}  \Upsilon(k_{\fk{l}_1},\cdots,k_{\fk{l}_{l}})=\sum_{k_{\fk{l}_{l+1}},\cdots,k_{\fk{l}_{2n_0}} }\mathbf{1}_{\sum_{i}\iota_{\fk{l}_i}k_{\fk{l}_i}=0 }\cdot \frac{\psi_{2s}(\vec{k}_{\fk{r}})}{\Omega(\vec{k}_{\fk{r}})}
			\Big(\prod_{j=1}^{l}\frac{1}{\langle k_{\fk{l}_j}\rangle^{s}}\Big)
			\Big(\prod_{j=l+1}^{2n_0}(a_{k_{\fk{l}_j}}^{(\fk{l}_j)})^{\iota_{\fk{l}_j}} \Big).
		\end{align}
		To proceed on, we distinct two cases:
		
		\noindent
		$\bullet${\bf Case 1: 
			$l\geq 3$}	
		
		\noindent
		$\bullet${\bf Subcase 1.1 : $\# X\geq 2$. } 
		In this case, we have automatically that $l\geq 2(\# X)\geq 4$. Consequently, $N_{\fk{l}_3},N_{\fk{l}_4}> N_{\fk{l}_1}^{1-\delta}$.
		By Cauchy-Schwarz, for fixed $k_{\fk{l}_j}, 1\leq j\leq l, \fk{l}_j\notin X\cup Y$, we have
		\begin{align*}
			&\sum_{\substack{ k_{\fk{l}}=k_{\sigma(\fk{l})}\\
					\fk{l}\in X }}|\Upsilon(k_{\fk{l}_1},\cdots,k_{\fk{l}_{l}})|\\
			\leq &\Big(\prod_{j=1}^{l}N_{\fk{l}_j}^{-s}\Big)\Big(\prod_{j=l+1}^{2n_0}\|a_{k_{\fk{l}_j}}^{(\fk{l}_j)}\|_{l^2}\Big)
			\Big(\sum_{k_{\fk{l}_{l+1}},\cdots,k_{\fk{l}_{2n_0}}}\Big( 
			\sum_{\substack{ k_{\fk{l}}=k_{\sigma(\fk{l})} \\
					\fk{l}\in X } }
			\mathbf{1}_{\sum_{i\geq l+1}\iota_{\fk{l}_i}k_{\fk{l}_i}=-\sum_{\fk{l}\notin X\cup Y}\iota_{\fk{l}}k_{\fk{l}}}\Big|\frac{\psi_{2s}(\vec{k}_{\fk{r}})}{\Omega(\vec{k}_{\fk{r}}) }\Big|\Big)^2
			\Big)^{\frac{1}{2}}.
		\end{align*}
		Therefore,
		\begin{align}\label{SumUpsilon} 
			&\sum_{\substack{ k_{\fk{l}_i}:\fk{l}_i\notin X\cup Y\\
					1\leq i\leq l   } } \Big(\sum_{
				k_{\fk{l}}=k_{\sigma(\fk{l})},\fk{l}\in X   } |\Upsilon(k_{\fk{l}_1},\cdots,k_{\fk{l}_{l}})|
			\Big)^2\notag  \\
			\leq &
			\Big(\prod_{j=1}^{l}N_{\fk{l}_j}^{-2s}\Big)
			\Big(\prod_{j=l+1}^{2n_0}\|a_{k_{\fk{l}_j}}^{(\fk{l}_j)}\|_{l^2}^2\Big)
			\sum_{\substack{ k_{\fk{l}_i}:\fk{l}_i\notin X\cup Y\\
					1\leq i\leq l  } }\sum_{k_{\fk{l}_{l+1}},\cdots,k_{\fk{l}_{2n_0}}}
			\Big(\sum_{ k_{\fk{l}}=k_{\sigma(\fk{l})},
				\fk{l}\in X }
			\mathbf{1}_{\sum_{i\geq l+1}\iota_{\fk{l}_i}k_{\fk{l}_i}=-\sum_{\fk{l}\notin X\cup Y}\iota_{\fk{l}}k_{\fk{l}}}\Big|\frac{\psi_{2s}(\vec{k}_{\fk{r}})}{\Omega(\vec{k}_{\fk{r}}) }\Big|\Big)^2\notag \\
			\lesssim 
			& N_{\fk{l}_1}^{-4s-2s(1-\delta)(l-2)}\Big(\prod_{j=l+1}^{2n_0}\|a_{k_{\fk{l}_j}}^{(\fk{l}_j)}\|_{l^2}^2N_{\fk{l}_j}^2\Big)\sup_{k_{\fk{l}_j}:j\geq l+1}	\sum_{\substack{ k_{\fk{l}_i},\fk{l}_i\notin X\cup Y, 1\leq i\leq l\\
					\sum_{\fk{l}\notin X\cup Y}\iota_{\fk{l}}k_{\fk{l}}=0
			} }
			\Big(
			\sum_{
				k_{\fk{l}}=k_{\sigma(\fk{l})},\fk{l}\in X
			}\Big|\frac{\psi_{2s}(\vec{k}_{\fk{r}})}{\Omega(\vec{k}_{\fk{r}})}\Big|
			\Big)^2.
		\end{align}
		For fixed $k_{\fk{l}_j}, j\geq l+1$, since
		\begin{align}\label{lnotinXY} 
			\sum_{\substack{ k_{\fk{l}_i},\fk{l}_i\notin X\cup Y, 1\leq i\leq l\\
					\sum_{\fk{l}\notin X\cup Y}\iota_{\fk{l}}k_{\fk{l}}=0
			} }1\lesssim \Big(\prod_{1\leq j\leq l,\fk{l}_j\notin X\cup Y }N_{\fk{l}_j}^2\Big)\big(\max_{1\leq j\leq l,\fk{l}_j\notin X\cup Y } N_{\fk{l}_j}^{-2}\big),
		\end{align}
		we obtain that
		\begin{align}\label{sumIII} 
			\text{r.h.s. \eqref{SumUpsilon}}  
			\lesssim &N_{\fk{l}_1}^{-2sl+2s\delta(l-2)}\cdot N_{\fk{l}_1}^{2(l-2\# X-1 )}\Big(\prod_{j=l+1}^{2n_0}\|a_{k_{\fk{l}_j}}^{(\fk{l}_j)}\|_{h^1}^2\Big)\cdot \Big(\sup_{\substack{ 1\leq j\leq l,\fk{l}_j\notin X\cup Y\\
					\sum_{\fk{l}\notin X\cup Y}\iota_{\fk{l}}k_{\fk{l}}=0 } }\mathrm{II}\Big),
		\end{align}
		where
		$$ \mathrm{II}:=\Big(\sum_{|\kappa|\lesssim N_{\fk{l}_1}N_{\fk{l}_3}}\sum_{
			k_{\fk{l}}=k_{\sigma(\fk{l}) },\fk{l}\in X
		}  \mathbf{1}_{\Omega(\vec{k}_{\fk{r}})=\kappa} \frac{|\psi_{2s}(\vec{k}_{\fk{r}})|}{\langle\kappa\rangle}  \Big)^2.
		$$
		
		To deal with the counting related to II, recall that $X\subset\mathcal{L}_{\fk{r}_1}=\{\fk{n}_1,\cdots,\fk{n}_{2l_0-1} \}$. 
		Set  $$\mathbf{a}=\sum_{k_{\fk{n}_i}: \fk{n}_i\neq \fk{l}_{i_0} }\iota_{\fk{n}_j}k_{\fk{n}_j}.$$ Since $\fk{l}_{i_0}\in X$ has parent $\fk{r}_1$ and it is paired with $\fk{l}_{j_0}\in\mathcal{L}_{\fk{r}}$, then for fixed $k_{\fk{l}_j}, j\neq i_0,j_0 $, the constraint $\Omega(\vec{k}_{\fk{r}})=\kappa$ takes the form :
		$$ |k_{\fk{l}_{i_0}}-\mathbf{a}|^2-|k_{\fk{l}_{i_0}}|^2=\mathrm{const.}
		$$
		Then if $\mathbf{a}\neq 0$, the number of $k_{\fk{l}_{i_0}}$ is bounded by $O(N_{\fk{l}_{i_0}})$, hence the corresponding contribution in $\mathrm{II}$ is bounded by (recall that $N_{\fk{l}_3}> N_{\fk{l}_1}^{1-\delta}>N_{\fk{l}_1}^{1/2}$)
		$$ N_{\fk{l}_1}^{4(s-1)}(N_{\fk{l}_{i_0}}^{-4}+ N_{\fk{l}_{i_0}}^{-2+}N_{\fk{l}_3}^4)\cdot \Big(\prod_{\fk{l}\in X}N_{\fk{l}}^{4}\Big)\lesssim N_{\fk{l}_1}^{4(s-1)+2+}N_{\fk{l}_3}^4\cdot N_{\fk{l}_1}^{4(\# X-1)}.
		$$
		If $\mathbf{a}=0$, 
		we argue differently.
		Note that $X\subset\{\fk{n}_1,\cdots,\fk{n}_{2l_0-1} \}$, and the constraint $\mathbf{a}=0$ allows us to fix one frequency $k_{\fk{n}_{i_*}}$ such that $N_{\fk{n}_{i_*}}\sim \max_{\fk{n}_i\in X\setminus\{\fk{l}_{i_0}\} }N_{\fk{n}_i}$ (this is possible since $X\setminus \{\fk{l}_{i_0} \}\neq \emptyset$). 
		Hence the contribution in $\mathrm{II}$ with the additional constraint $\mathbf{a}=0$ can be bounded by
		$$ 
		N_{\fk{l}_1}^{4(s-1)}(N_{\fk{l}_3}^4+1)\Big(\prod_{\fk{l}\in X}N_{\fk{l}}^4\Big)\cdot \big(\max_{\fk{l}\in X\setminus\{\fk{l}_{i_0}\} }N_{\fk{l}}^{-4}\big)\lesssim N_{\fk{l}_1}^{4(s-1)}N_{\fk{l}_3}^4\cdot N_{\fk{l}_1}^{4(\# X-1)}.
		$$
		Plugging into \eqref{sumIII}, we obtain that
		$$ \mathrm{r.h.s.} \eqref{sumIII}\lesssim N_{\fk{l}_1}^{-2(l-2)(s-1)+2s\delta (l-2)}\prod_{j=l+1}^{2n_0}\|a\|_{h^1}^2,
		$$
		and the corresponding contribution in \eqref{LpI} is bounded by
		\begin{align}\label{outputcase1}  p^{\frac{l}{2}}N_{\fk{l}_1}^{-(l-2)(s-1-s\delta)}\prod_{j=l+1}^{2n_0}\|a\|_{h^1}\leq p^{\frac{l}{2}}N_{\fk{l}_1}^{-(l-2)(s-1-s\delta)}\lambda^{2n_0-l}.
		\end{align}
		
		\noi
		$\bullet${\bf Subcase 1.2 $\#X=1 $ } Then only dominated leaves $\fk{l}_{i_0}$ and $\fk{l}_{j_0}$ are paired, and
		the analysis is essentially the same. 
		Following the same line of argument as \eqref{SumUpsilon},\eqref{lnotinXY} and \eqref{sumIII},
		\begin{align}\label{sumIII'} 
			\text{r.h.s. \eqref{LpI}}  
			\lesssim &
			p^{\frac{l}{2}}\Big(\prod_{j=1}^{l}N_{\fk{l}_j}^{-s}\Big)\Big(\prod_{j=l+1}^{2n_0}\|a_{k_{\fk{l}_j}}^{(\fk{l}_j)}\|_{l^2} \Big)\Big(\sum_{\substack{k_{\fk{l}_j}:1\leq j\leq 2n_0, j\neq i_0,j_0\\
					\text{no pairing among }k_{\fk{l}_i},1\leq i\leq l, i\neq i_0,j_0 } }\mathbf{1}_{\sum_{j\neq i_0,j_0 }\iota_{\fk{l}_j}k_{\fk{l}_j}=0 }\cdot\mathrm{II}'\Big)^{\frac{1}{2}},
		\end{align}
		where
		$$ \mathrm{II}':=\Big(\sum_{k_{\fk{l}_{i_0}}=k_{\fk{l}_{j_0}}
		}  \Big|\frac{\psi_{2s}(\vec{k}_{\fk{r}})}{\Omega(\vec{k}_{\fk{r}})}\Big|  \Big)^2.
		$$
		Again we set $\mathbf{a}=\sum_{\fk{n}_i:\fk{n}_i\neq \fk{l}_{i_0} }\iota_{\fk{n}_i}k_{\fk{n}_i}$. Without loss of generality, we assume that $\fk{l}_{i_0}=\fk{n}_1$ is the child of the branching node $\fk{r}_1$ and its paired leaf $\fk{l}_{j_0}=\fk{r}_2$ is the child of the root $\fk{r}$. If $\mathbf{a}\neq 0$, the same analysis using two-vector counting yields
		$$\mathrm{(r.h.s. \eqref{sumIII'})}\lesssim p^{\frac{l}{2}}N_{\fk{l}_1}^{-(l-2)(s-1-s\delta)}\lambda^{2n_0-l}.
		$$
		Now assume that $\mathbf{a}=0$. Recall that $\fk{r}_1,\cdots,\fk{r}_{2m_0}$ are children of the root $\fk{r}$, we will instead using this additional constraint which breaks into two relations :
		$$ \sum_{j=3}^{2m_0}\iota_{\fk{r}_j}k_{\fk{r}_j}=0,\quad \sum_{j=2}^{2l_0-1}\iota_{\fk{n}_j}k_{\fk{n}_j}=0.
		$$
		Since $\Omega(\vec{k}_{\fk{r}}),\psi_{2s}(\vec{k}_{\fk{r}})$ depend only on $k_{\fk{r}_i}, 3\leq i\leq 2m_0$, so $\mathrm{II}'$ is trivially bounded by $N_{\fk{l}_1}^4\Big|\frac{\psi_{2s}(\vec{k}_{\fk{r}})}{\Omega(\vec{k}_{\fk{r}})}\Big|^2$ for fixed $k_{\fk{l}_j}, j\neq i_0,j_0$.
		Denote by $M_1\geq M_2\geq\cdots M_{2m_0-2}$ the non-increasing rearrangement of $N_{\fk{r}_3},\cdots, N_{\fk{r}_{2m_0}}$.  Denote $\mathcal{T}'$ the simple tree obtained by keeping the leaves $\fk{r}_{3},\cdots,\fk{r}_{2m_0}$ and the root $\fk{r}$. 
		 By \eqref{l2cor3.3}, we obtain that
		\begin{align*}
			\Big(\sum_{\substack{k_{\fk{l}_j}:1\leq j\leq 2n_0, j\neq i_0,j_0\\
					\text{no pairing among }k_{\fk{l}_i},1\leq i\leq l, i\neq i_0,j_0 } }&\mathbf{1}_{\sum_{j\neq i_0,j_0 }\iota_{\fk{l}_j}k_{\fk{l}_j}=0 }\cdot\mathrm{II}'\Big)^{\frac{1}{2}}
			\lesssim  N_{\fk{l}_1}^2\frac{\prod_{\fk{n}_j: 2\leq j\leq 2l_0-1}N_{\fk{n}_j}}{\max_{\fk{n}_j: 2\leq j\leq 2l_0-1}N_{\fk{n}_j}}\cdot  \Big(\sum_{k_{\fk{r}_i},\fk{r}_i\in\mathcal{T}'  }\Big|\frac{\psi_{2s}(\vec{k}_{\fk{r}}) }{\Omega(\vec{k}_{\fk{r}})} \Big|^2 \Big)^{\frac{1}{2}}\\
			\lesssim &N_{\fk{l}_1}^2\frac{\prod_{\fk{n}_j: 2\leq j\leq 2l_0-1}N_{\fk{n}_j}}{\max_{\fk{n}_j: 2\leq j\leq 2l_0-1}N_{\fk{n}_j}}\cdot
			M_1^{2(s-1)+}(M_2M_3+M_3^{\frac{5}{2}}M_2^{\frac{1}{2}})\Big(\prod_{j=4}^{2m_0-2}M_j\Big) \\
			\leq &N_{\fk{l}_1}^2N_{\fk{l}_{i_1}}^{2(s-1)+1+}\cdot M_2M_3\cdot \frac{\prod_{\fk{n}_j: 2\leq j\leq 2l_0-1}N_{\fk{n}_j}
			}{\max_{\fk{n}_j: 2\leq j\leq 2l_0-1}N_{\fk{n}_j}}\cdot	\prod_{j=4}^{2m_0-2}M_j	, 
		\end{align*}
		where to the last step, we use the definition of $i_1$ in \eqref{i1j1}, hence $M_1\sim M_2\leq N_{\fk{l}_{i_1}}$, and $M_3^{5/2}M_2^{1/2}+M_2M_3\leq M_2^2M_3\lesssim N_{\fk{l}_{i_1}}M_2M_3 $.
		We crudely estimate
		$$ M_2M_3\cdot \frac{\prod_{\fk{n}_j: 2\leq j\leq 2l_0-1}N_{\fk{n}_j}
		}{\max_{\fk{n}_j: 2\leq j\leq 2l_0-1}N_{\fk{n}_j}}\cdot	\prod_{j=4}^{2m_0-2}M_j\leq \frac{\prod_{j=1}^{l}N_{\fk{l}_j}}{N_{\fk{l}_{i_0}}N_{\fk{l}_{j_0}}N_{\fk{l}_{i_1}}}
		\prod_{i=l+1}^{2n_0}N_{\fk{l}_i}
		\leq N_{\fk{l}_1}^{2\delta}N_{\fk{l}_3}^{l-3}\prod_{i=l+1}^{2n_0}N_{\fk{l}_i},
		$$
		where the last inequality comes from the fact that $N_{\fk{l}_{i_0}}N_{\fk{l}_{j_0}}N_{\fk{l}_{i_1}}\geq N_{\fk{l}_{i_0}}N_{\fk{l}_{j_0}}N_{\fk{l}_3}$ and $N_{\fk{l}_{i_0}}=N_{\fk{l}_{j_0}}>N_{\fk{l}_1}^{1-\delta}$.
		Plugging into \eqref{sumIII'},  we deduce that
		\begin{align} 
			\mathrm{(r.h.s. \eqref{sumIII'})}\lesssim &p^{\frac{l}{2}}N_{\fk{l}_1}^{-ls+s\delta(l-2)}\Big(\prod_{j=l+1}^{2n_0}\|a_{k_{\fk{l}_j}}^{(\fk{l}_j)}\|_{h^1}\Big)\cdot N_{\fk{l}_1}^{2+2\delta}N_{\fk{l}_{i_1}}^{2(s-1)+1}\cdot N_{\fk{l}_3}^{l-3}\notag \\
			\lesssim & p^{\frac{l}{2}}N_{\fk{l}_1}^{-(l-2)(s-1-s\delta)+2\delta} \lambda^{2n_0-l}. \label{outputcase2} 
		\end{align}
		
		In summary, we proved :
		\[
		\|\mathcal T_{e,X,Y}^\sigma\|_{L^p}
		\;\lesssim\;
		p^{\frac l2}\,N_{\fk{l}_1}^{-(l-2)(s-1-s\delta)}\,
		\lambda^{2n_0-l},
		\quad l\ge3.
		\]

		\noi
		$\bullet${\bf Case 2: $l=2$} 
		By assumption, $\#X=\#Y=1$.
		
		In this case, automatically we have $i_0=1, j_0=2$, namely $\fk{l}_1,\fk{l}_2$ are paired, and they belong to two different generations. Since
		$ \sum_{j\geq 3}\iota_{\fk{l}_j}k_{\fk{l}_j}=0,
		$
		we have $N_{\fk{l}_4}\sim N_{\fk{l}_3}\leq N_{\fk{l}_1}^{1-\delta}$, and
		the contribution in \eqref{LpI} is bounded by
		\begin{align}\label{LpI'}
			&p\sum_{k_{\fk{l}_1}=k_{\fk{l}_2}} \frac{1}{\langle k_{\fk{l}_1}\rangle^{2s}} \sum_{k_{\fk{l}_j}:j\geq 3 }\mathbf{1}_{\sum_{j\geq 3}\iota_{\fk{l}_j}k_{\fk{l}_j}=0 }
			\Big|\frac{\psi_{2s}(\vec{k}_{\fk{r}})}{
				\Omega(\vec{k}_{\fk{r}}) }\Big|\prod_{j=3}^{2n_0}|a_{k_{\fk{l}_j}}^{(\fk{l}_j)}|\notag \\
			\lesssim & pN_{\fk{l}_1}^{-2s}\sum_{k_{\fk{l}_j}: j\geq 3 }\mathbf{1}_{\sum_{j\geq 3}\iota_{\fk{l}_j}k_{\fk{l}_j}=0
			}\prod_{j=3}^{2n_0}|a_{k_{\fk{l}_j}}^{(\fk{l}_j)}|\sum_{1\leq|\kappa|\lesssim N_{\fk{l}_1}N_{\fk{l}_3}} \sum_{k_{\fk{l}_1}=k_{\fk{l}_2}}\mathbf{1}_{\Omega(\vec{k}_{\fk{r}})=\kappa}\frac{N_{\fk{l}_1}^{2(s-1)}(N_{\fk{l}_3}^2+|\kappa|) }{\langle\kappa\rangle}.
		\end{align}
		Without loss of generality, we assume that $\fk{l}_1=\fk{n}_1\in \mathcal{L}_{\fk{r}_1}$ and $\fk{l}_2=\fk{r}_2\in\mathcal{L}_{\fk{r}}$.
		Set  $$\mathbf{a}=\sum_{j=2}^{2l_0-1}\iota_{\fk{n}_j}k_{\fk{n}_j}.$$ Fix $k_{\fk{l}_j}, j\geq 3$, if $\mathbf{a}\neq 0$, then as in the analysis of the previous case, the number of $k_{\fk{l}_1}$ such that $\Omega(\vec{k}_{\fk{r}})=\kappa$ is $O(N_{\fk{l}_1})$, we bound the right hand side of \eqref{LpI'} by
		\begin{align*} &pN_{\fk{l}_1}^{-2s}\|a_{k_{\fk{l}_3}}^{(\fk{l}_3)}\|_{l^2}
			\|a_{k_{\fk{l}_4}}^{(\fk{l}_4)}\|_{l^2}\Big(\prod_{j=5}^{2n_0}\|a_{k_{\fk{l}_j}}^{(\fk{l}_j)}\|_{h^1}\Big)\cdot N_{\fk{l}_1}^{2s-2}\big(N_{\fk{l}_1}^{1+}N_{\fk{l}_3}^{2}+N_{\fk{l}_1}^2\big)\notag \\
			\lesssim &
			p\big(N_{\fk{l}_1}^{-1+}+N_{\fk{l}_3}^{-2}\big)\Big(\prod_{j=3}^{2n_0}\|a_{k_{\fk{l}_j}}^{(\fk{l}_j)}\|_{h^1}\Big)\leq p(N_{\fk{l}_1}^{-1+}+N_{\fk{l}_3}^{-2})\lambda^{2n_0-2}.
		\end{align*} 
		If $\mathbf{a}=0$, then $k_{\fk{r}_1}=k_{\fk{r}_2}$, and
		$$ |\psi_{2s}(\vec{k}_{\fk{r}})|\lesssim N_{\fk{l}_3}^{2(s-1)}(N_{\fk{l}_5}^2+|\Omega(\vec{k}_{\fk{r}})|),
		$$
		where $|\Omega(\vec{k}_{\fk{r}})|\lesssim N_{\fk{l}_3}N_{\fk{l}_5}$. The right hand side of \eqref{LpI'} can be controlled crudely by
		\begin{align*}
			pN_{\fk{l}_1}^{-2(s-1)}N_{\fk{l}_3}^{2(s-1)}\prod_{j=3}^{2n_0}\|a_{k_{\fk{l}_j}}^{(\fk{l}_j)}\|_{h^1}\leq pN_{\fk{l}_1}^{-2(s-1)\delta}\lambda^{2n_0-2}.
		\end{align*}
		Put the two possibility together, the right hand side of \eqref{LpI'} is bounded by
		\begin{align}\label{Case2-1} 
			p(N_{\fk{l}_1}^{-2(s-1)\delta}+N_{\fk{l}_3}^{-2} )\lambda^{2n_0-2}.
		\end{align}
		
		\noi
		$\bullet${\bf Interpolation :}
		Putting estimates \eqref{outputcase1},\eqref{Case2-1} and  \eqref{outputcase2} together, we deduce that
		\begin{align*}
			(\mathrm{r.h.s. of \eqref{LpI}})\leq 
			\begin{cases}
				\; p(N_{\fk{l}_1}^{-2(s-1)\delta}+N_{\fk{l}_3}^{-2} )\lambda^{2n_0-2},\text{ when } l=2,\\
				\; p^{\frac{l}{2}}N_{\fk{l}_1}^{-(l-2)(s-1-s\delta)+2\delta}\lambda^{2n_0-l},\text{ when } l\geq 3.	
			\end{cases}
		\end{align*}
		Interpolating with the bound \eqref{I}, we deduce that if $l\geq 3$, for any $p\geq 2,\epsilon\in(0,1)$,
		$$ \|\mathcal{T}_{e,X,Y}^{\sigma}\mathbf{1}_{\|\phi\|_{H^1}\leq \lambda}\|_{L_{d\mu_s}^p}\lesssim_{\epsilon} p^{1-\epsilon}\cdot \lambda^{2n_0}N_{\fk{l}_1}^{-\frac{l-2}{l}(2-\frac{2l\delta}{l-2}-2s\delta)+C\epsilon
		}\leq 
		p^{1-\epsilon}\cdot \lambda^{2n_0}N_{\fk{l}_1}^{-\frac{l-2}{l} },
		$$ 
		thanks to the choice of $\delta<\frac{1}{100(s+n_0)}$.
		If $l=2$, we deduce that for any $\epsilon\in(0,1)$,
		$$ \|\mathcal{T}_{e,X,Y}^{\sigma}\mathbf{1}_{\|\phi\|_{H^1}\leq \lambda}\|_{L_{d\mu_s}^p}\lesssim_{\epsilon} p^{1-\frac{\epsilon}{2s}}\lambda^{2n_0}N_{\fk{l}_1}^{\epsilon}\big(N_{\fk{l}_1}^{-2(s-1)\delta(1-\frac{\epsilon}{2s})}+N_{\fk{l}_3}^{-2(1-\frac{\epsilon}{2s})}\big).
		$$
		The proof of Proposition \ref{crosspaired} is now complete.
	\end{proof}

	\section{Modified energy estimate I: Cross-pairing contributions }\label{singulartree} 

	Recall the definition \eqref{ST}  of the cross-pairing term $\mathcal{S}_{\fk{l}',\fk{l}''}(\mathcal{T})(v)$ with paired leaves $\fk{l}',\fk{l}''$ in two different generations. We prove in this section :
	\begin{prop}\label{singularterm} 
		Let $\mathcal{T}$ be an expanded branching tree of scale 2 with root $\fk{r}$ and the other branching node $\fk{n}_0$. Assume that $\#\mathcal{L}_{\fk{r}}=\#\mathcal{L}_{\fk{n}_0}=n_0=2m-1$. Let $\fk{l}'\in\mathcal{L}_{\fk{r}},\fk{l}''\in\mathcal{L}_{\fk{n}_0}$, then there exists $\epsilon_0\in(0,1)$ such that for any $p\geq 2$, $\lambda\geq 1$,
		$$ \|\Im\mathcal{S}_{\fk{l}',\fk{l}''}(\mathcal{T})(v)\mathbf{1}_{\|v\|_{H^{1} }\leq \lambda}\|_{L^p(d\mu_s)}\lesssim p^{1-\epsilon_0}\lambda^{2n_0}.
		$$
		where the implicit constant is independent of $p$.
	\end{prop}
	\subsection{Key cancellation identity}
	
	By symmetry, we may only consider the  trees $\mathcal{T}^+,\mathcal{T}^-$, where for $\mathcal{T}^+$, its non-root node $\fk{n}_0$ has the sign $\iota_{\fk{n}_0}=+$, and for $\mathcal{T}^-$, its non-root node $\fk{n}_0$ has the sign $\iota_{\fk{n}_0}=-$. From the symmetry of indices, there are four possible singular paired configurations $\mathcal{S}_{\fk{l}',\fk{l}''}^{-+}(\mathcal{T}^+)$, $\mathcal{S}_{\fk{l}',\fk{l}''}^{+-}(\mathcal{T}^+)$,
	$\mathcal{S}_{\fk{l}',\fk{l}''}^{+-}(\mathcal{T}^-)$, $\mathcal{S}_{\fk{l}',\fk{l}''}^{-+}(\mathcal{T}^-)$, with paired leaves $\fk{l}',\fk{l}''$,
	as shown in the graphs below : 
	\vspace{0.3cm}

	\begin{tikzpicture}[scale=1]
		\tikzset{level distance=25pt};
		\tikzset{sibling distance=20pt};
		[
		level 1/.style={sibling distance=10pt},
		level 2/.style={level post sep=10pt}, level distance=10pt,
		] 
		\draw (1,0) circle (3.0pt);
		\draw (0,0).. controls (0.3,0.3) and (0.8,0.3) .. (0.95,0);
		\node at (0,0) [left] {$\mathfrak{r}$};
		\node at (1,0) [right] {$\mathfrak{r}'$};
		\coordinate (root) {}  [fill]  circle (3.0pt)
		child {  
			[fill]  circle (3.0pt)
			child {  [fill]  circle (3.0pt)
				node[left]{$\mathfrak{l}''$} 
			} 
			child { circle (3.0pt) 
			}
			child { [fill] circle (3.0pt)         
			}
			child {    circle (3.0pt) 
			}
			child {  [fill]  circle (3.0pt) 
			}
		}
		child {    circle (3.0pt) 
		}
		child {  [fill]  circle (3.0pt) 
		}
		child {    circle (3.0pt) 
		}
		child {  [fill]  circle (3.0pt) 
		}
		;
		\node at (root-1)[left] {$\mathfrak{n}_0$};
		\node at (root-2)[right] {$\mathfrak{l}'$};
		\node at (0,-3) {Configuration 1: $\mathcal{S}^{-+}_{\mathfrak{l}',\mathfrak{l}''}(\mathcal{T}^+)$};
	\end{tikzpicture}
	\hspace{2.5cm} 
	\begin{tikzpicture}[scale=1]
		\tikzset{level distance=25pt};
		\tikzset{sibling distance=20pt}; 
		[
		level 1/.style={sibling distance=2em},
		level 2/.style={sibling distance=0.5em}, level distance=0.5cm,
		] 
		\draw (1,0) circle (3.0pt);
		\draw (0,0).. controls (0.3,0.3) and (0.8,0.3) .. (0.95,0);
		\node at (0,0) [left] {$\mathfrak{r}$};
		\node at (1,0) [right] {$\mathfrak{r}'$};
		\coordinate (root) {}  [fill]  circle (3.0pt)
		child {  [fill]  circle (3.0pt)
			child {  [fill]  circle (3.0pt)
			}
			child { circle (3.0pt) 
				node[left]{$\mathfrak{l}''$}
			}
			child { [fill] circle (3.0pt) 
			}
			child {  circle (3.0pt)
			}
			child { [fill] circle (3.0pt)         
			}
		}
		child {    circle (3.0pt) 
		}
		child { [fill] circle (3.0pt)
		}
		child {  circle (3.0pt)
		}
		child {  [fill]  circle (3.0pt) 
		}
		;
		\node at (root-1)[left] {$\mathfrak{n}_0$};
		\node at (root-3)[right] {$\mathfrak{l}'$};
		\node at (0,-3) {Configuration 2: $\mathcal{S}^{+-}_{\mathfrak{l}',\mathfrak{l}''}(\mathcal{T}^+)$};	
	\end{tikzpicture}

	\begin{tikzpicture}[scale=1]
		\tikzset{level distance=25pt};
		\tikzset{sibling distance=20pt}; 
		[
		level 1/.style={sibling distance=3em},
		level 2/.style={sibling distance=1.5em}, level distance=1cm,
		] 
		\draw (1,0) circle (3.0pt);
		\draw (0,0).. controls (0.3,0.3) and (0.8,0.3) .. (0.95,0);
		\node at (0,0) [left] {$\mathfrak{r}$};
		\node at (1,0) [right] {$\mathfrak{r}'$};
		\coordinate (root) {}  [fill]  circle (3.0pt)
		child {  [fill]  circle (3.0pt) 
			node [left] {$\fk{l}'$}
		}
		child {    circle (3.0pt)
			child {    circle (3.0pt)
				node [right] {$\fk{l}''$}
			}
			child { [fill] circle (3.0pt) 
			}
			child {  circle (3.0pt)         
			}
			child { [fill] circle (3.0pt)
			}
			child {  circle (3.0pt) 
			}
		}
		child { [fill] circle (3.0pt)
		}
		child {  circle (3.0pt) 
		}
		child {  [fill]  circle (3.0pt) 
		}
		;
		\node at (root-2)[left] {$\mathfrak{n}_0$};
		\node at (0,-3) {Configuration 3: $\mathcal{S}^{+-}_{\mathfrak{l}',\mathfrak{l}''}(\mathcal{T}^-)$};
	\end{tikzpicture}
	\hspace{2.5cm} 
	\begin{tikzpicture}[scale=1] 
		\tikzset{level distance=25pt};
		\tikzset{sibling distance=20pt}; 
		[
		level 1/.style={sibling distance=3em},
		level 2/.style={sibling distance=1.5em}, level distance=1cm,
		] 
		\draw (1,0) circle (3.0pt);
		\draw (0,0).. controls (0.3,0.3) and (0.8,0.3) .. (0.95,0);
		\node at (0,0) [left] {$\mathfrak{r}$};
		\node at (1,0) [right] {$\mathfrak{r}'$};
		\coordinate (root) {}  [fill]  circle (3.0pt)
		child {  [fill]  circle (3.0pt) 
		}
		child {    circle (3.0pt)
			child {    circle (3.0pt)
			}
			child { [fill] circle (3.0pt) 
				node[left] {$\fk{l}''$}
			}
			child {  circle (3.0pt)         
			}
			child { [fill] circle (3.0pt)
			}
			child {  circle (3.0pt) 
			}
		}
		child { [fill] circle (3.0pt)
		}
		child {  circle (3.0pt) 
			node[right] {$\fk{l}'$}
		}
		child {  [fill]  circle (3.0pt) 
		}
		;
		\node at (root-2)[left] {$\mathfrak{n}_0$};
		\node at (0,-3) {Configuration 4: $\mathcal{S}^{-+}_{\mathfrak{l}',\mathfrak{l}''}(\mathcal{T}^-)$};
	\end{tikzpicture} 
	
	Without loss of generality, for $\mathcal{T}\in\{\mathcal{T}^+,\mathcal{T}^-\}$, below we assume that the leaf $\fk{l}'$ belongs to the first generation (hence $\fk{l}''$ belongs to the second generation). For such a tree $\mathcal{T}$, 
	denote by
	\begin{align}\label{Psi}  \Psi(\vec{k}_{\fk{r}}):= \frac{\psi_{2s}(\vec{k}_{\fk{r}}) }{\Omega(\vec{k}_{\fk{r}})}-\frac{\iota_{\fk{n}_0}|k_{\fk{n}_0}|^{2s}+\iota_{\fk{l}'}|k_{\fk{l}'}|^{2s} }{\iota_{\fk{n}_0}|k_{\fk{n}_0}|^2+\iota_{\fk{l}'}|k_{\fk{l}'}|^2}.
	\end{align}
	
	The key cancellation identity is as follows:
	\begin{lem}\label{keycancellation}
		The following identity holds:
		\begin{align}\label{identity:cancellationkey}
			\Im \mathcal{S}_{\fk{l}',\fk{l}''}(\mathcal{T})=\Im\sum_{(k_{\fk{n}})\in\Lambda_{\fk{l}',\fk{l}''}   } \Psi(\vec{k}_{\fk{r}})\prod_{\fk{l}\in\mathcal{L}}v_{k_{\fk{l}}}^{\iota_{\fk{l}}}.
		\end{align}
	\end{lem}
	\begin{proof}
		Since $k_{\fk{l}'}=k_{\fk{l}''}, \iota_{\fk{l}'}=-\iota_{\fk{l}'}$, the constraint $\sum_{\fk{l}\in\mathcal{L}}\iota_{\fk{l}}k_{\fk{l}}=0$ implies
		$$ 
		\iota_{\fk{n}_0}k_{\fk{n}_0}+\iota_{\fk{l}'}k_{\fk{l}'}=\iota_{\fk{n}_0}k_{\fk{n}_0}-\iota_{\fk{l}''}k_{\fk{l}''}
		=\sum_{\fk{l}\in\mathcal{L}_{\fk{n}_0},\;\fk{l}\neq \fk{l}'' }\iota_{\fk{l}}k_{\fk{l}}=- \sum_{\fk{l}\in\mathcal{L}_{\fk{r}},\;\fk{l}\neq \fk{l}'}\iota_{\fk{l}}k_{\fk{l}}.
		$$
		Since $\mathcal{L}_{\fk{r}}$ is isomorphis to $\mathcal{L}_{\fk{n_0}}$, and $\fk{l}'$ are paired with $\fk{l}''$, there is a bijection mapping each leaf $\fk{l}$ in $\mathcal{L}_{\fk{r}}\setminus \{\fk{l}'\}$ to a leaf in $\mathcal{L}_{\fk{n}_0}\setminus\{\fk{l}''\}$ with the opposite sign.  Hence
		\begin{align*}
			&\sum_{(k_{\fk{n}})\in\Lambda_{\fk{l}',\fk{l}''}   } \Big(\frac{\psi_{2s}(\vec{k}_{\fk{r}})}{\Omega(\vec{k}_{\fk{r}})}-\Psi(\vec{k}_{\fk{r}})\Big)\prod_{\fk{l}\in\mathcal{L}}v_{k_{\fk{l}}}^{\iota_{\fk{l}}}\\
			=&\sum_{k_{\fk{l}''},k_{\fk{n}_0}}\frac{\iota_{\fk{n}_0}|k_{\fk{n}_0}|^{2s}-\iota_{\fk{l}''}|k_{\fk{l}''}|^{2s} }{\iota_{\fk{n}_0}|k_{\fk{n}_0}|^2-\iota_{\fk{l}''}|k_{\fk{l}''}|^2} |v_{\fk{l}''}|^2\Big|\sum_{\fk{l}\in\mathcal{L}_{\fk{n}_0},\fk{l}\neq \fk{l}'' } 
			\mathbf{1}_{S(k_{\fk{n_0}}k_{\fk{l}''})}
			\prod_{\fk{l}\in\mathcal{L}_{\fk{n}_0},\fk{l}\neq\fk{l}'' }v_{k_{\fk{l}}}^{\iota_{\fk{l}}}  \Big|^2,
		\end{align*}
		which is real-valued, where
		\begin{align}\label{Sn0l''} S(k_{\fk{n}_0},k_{\fk{l}''}):=\Big\{k_{\fk{l}}: \fk{l}\in\mathcal{L}_{\fk{n}_0}\setminus\{ \fk{l}''\},\sum_{\fk{l}\in\mathcal{L}_{\fk{n}_0},\fk{l}\neq\fk{l}''}\iota_{\fk{l}}k_{\fk{l}}=\iota_{\fk{n}_0}k_{\fk{n}_0}-\iota_{\fk{l}''}k_{\fk{l}''}, \sum_{\fk{l}\in\mathcal{L}_{\fk{n}_0},\;\fk{l}\neq\fk{l}''}|k_{\fk{l}}|\leq |k_{\fk{n}_0}|^{\theta}+|k_{\fk{l}''}|^{\theta}  \Big\}
		\end{align}
		is a subset of decoration in $\Lambda_{\fk{l'},\fk{l}''}$ with fixed $k_{\fk{n}_0},k_{\fk{l}''}$. 
		This proves Lemma \ref{keycancellation}.
	\end{proof}
	
	\begin{lem}\label{boundPsi} 
		Recall that $\Psi(\vec{k}_{\fk{r}})$ is given by \eqref{Psi} and $\Lambda_{\fk{l}',\fk{l}''}$ is given by \eqref{Lambda}. We have:
		\begin{itemize}
			\item[$\mathrm{(i)}$] If $\iota_{\fk{n}_0}=-\iota_{\fk{l}'}$, then
			\begin{align*}
				|\Psi(\vec{k}_{\fk{r}})|\mathbf{1}_{\Lambda_{\fk{l}',\fk{l}''}}\lesssim \frac{|k_{\fk{n}_0}|^{2s-2} \lambda_3(\vec{k}_{\fk{r}})^2 }{|\Omega(\vec{k}_{\fk{r}})|}\mathbf{1}_{k_{\fk{n}_0}\neq k_{\fk{l}'} }+\frac{\lambda_3(\vec{k}_{\fk{r}})^{2s}}{|\Omega(\vec{k}_{\fk{r}})|}\mathbf{1}_{k_{\fk{n}_0}=k_{\fk{l}'} }. 
			\end{align*}
			\item[$\mathrm{(ii)}$] If $\iota_{\fk{n}_0}=\iota_{\fk{l}'}$, then
			$$ |\Psi(\vec{k}_{\fk{r}})|\mathbf{1}_{\Lambda_{\fk{l}',\fk{l}''}}\lesssim |k_{\fk{n}_0}|^{2s-4}\lambda_3(\vec{k}_{\fk{r}})^2,
			$$
		\end{itemize}
		where $\lambda_3(\vec{k}_{\fk{r}})$ is the third largest number of $k_{\fk{n}}$ among children and partner leave of the root $\fk{r}$, see \eqref{def:orderfrequency}. 
	\end{lem}
	\begin{proof}
		Set $\psi_{2s}^0:=\iota_{\fk{n}_0}|k_{\fk{n}_0}|^{2s}+\iota_{\fk{l}'}|k_{\fk{l}'}|^{2s}$ and $\Omega^0:=\iota_{\fk{n}_0}|k_{\fk{n}_0}|^{2}+\iota_{\fk{l}'}|k_{\fk{l}'}|^{2}$. We note that $|k_{\fk{n}_0}|\sim |k_{\fk{l}'}|$ on $\Lambda_{\fk{l}',\fk{l}''}$.
		
		First we deal with the case $\iota_{\fk{n}_0}=-\iota_{\fk{l}'}$. If $k_{\fk{n}_0}=k_{\fk{l}'}$, the estimate is trivial. Now we assume that $k_{\fk{n}_0}\neq k_{\fk{l}'}$. 
		We have
		\begin{align}\label{PsikT} 
			\Psi(\vec{k}_{\fk{r}})=&\frac{\psi_{2s}(\vec{k}_{\fk{r}})}{\Omega(\vec{k}_{\fk{r}})}-\frac{\psi_{2s}^0}{\Omega^0}=\frac{\psi_{2s}(\vec{k}_{\fk{r}})-\psi_{2s}^0  }{\Omega(\vec{k}_{\fk{r}}) } +\frac{\psi_{2s}^0}{\Omega^0}\cdot\frac{\Omega^0-\Omega(\vec{k}_{\fk{r}})
			}{\Omega(\vec{k}_{\fk{r}})}.
		\end{align}
		As $|\psi_{2s}^0/\Omega^0|\lesssim |k_{\fk{n}_0}|^{2s-2}$, the desired estimate follows directly.
		
		If $\iota_{\fk{n}_0}=\iota_{\fk{l}'}$, then $|\Omega(\vec{k}_{\fk{r}})|\sim |k_{\fk{n}_0}|^2$, from \eqref{PsikT}, it follows that
		$$ |\Psi(\vec{k}_{\fk{r}})|\lesssim \lambda_3(\vec{k}_{\fk{r}})^{2s}|k_{\fk{n}_0}|^{-2}+|k_{\fk{n}_0}|^{2s-4}\lambda_3(\vec{k}_{\fk{r}})^2\lesssim |k_{\fk{n}_0}|^{2s-4}\lambda_3(\vec{k}_{\fk{r}})^2,
		$$ 
		thanks to $s>2$.
		This completes the proof of Lemma \ref{boundPsi}.
	\end{proof}

	\subsection{Proof of Proposition \ref{singularterm}  }
	\begin{proof}
		By Lemma \ref{keycancellation}, we need to estimate the right hand side of \eqref{identity:cancellationkey}.  Recall that the adapted dyadic scales $N_{\fk{l}}$ are quasi-ordered as: $N_{\fk{l}_1}\geq N_{\fk{l}_2}\geq\cdots N_{\fk{l}_{2n_0}}$. Thanks to Lemma \ref{keycancellation}, we have
		$$ \Im\mathcal{S}_{\fk{l}',\fk{l}''}(\mathcal{T})(v)=\sum_{N_{\fk{l}} \text{ adapted } } \Im \mathcal{S}_{\fk{l}',\fk{l}''}(\mathcal{T};(N_{\fk{l}}))(v),
		$$
		where
		$$ \Im \mathcal{S}_{\fk{l}',\fk{l}''}(\mathcal{T};(N_{\fk{l}}))(v):=\Im \sum_{\substack{(k_{\fk{n}}\in\Lambda_{\fk{l}',\fk{l}''})\\ 
				N_{\fk{n}}\leq |k_{\fk{n}}|< 2N_{\fk{n}}  } }\Psi(\vec{k}_{\fk{r}})|v_{k_{\fk{l}'}}|^2\cdot\prod_{\fk{l}\in\mathcal{L}\setminus\{\fk{l}',\fk{l}'' \} }v_{k_{\fk{l}}}^{\iota_{\fk{l}}}.
		$$
		It suffices to show that for some $\epsilon>0$, 
		\begin{align}\label{subgoal:Lemma5.2}
			\|\Im \mathcal{S}_{\fk{l}',\fk{l}''}(\mathcal{T};(N_{\fk{l}}))(v)\mathbf{1}_{\|v\|_{H^1}\leq \lambda}\|_{L^p(d\mu_s)}\leq p^{1-\epsilon}N_{\fk{l}_1}^{-C\epsilon}\lambda^{2n_0}.
		\end{align}
		Under the measure $\mu_s$, $|v_{k_{\fk{l}'}}|^2=|g_{k_{\fk{l}'}}|^2/\langle k_{\fk{l}'}\rangle^{2s}$, since the remainder involving $(|g_{k_{\fk{l}'}}|^2-1)/\langle k_{\fk{l}'}^{2s}\rangle$ give better estimate, it suffices to prove that 
		\begin{align}\label{dyadicIm} 
			\mathrm{I}:=\sum_{(k_{\fk{n}})\in\Lambda_{\fk{l}',\fk{l}''}   } \frac{|\Psi(\vec{k}_{\fk{r}})|}{\langle k_{\fk{l}''}\rangle^{2s} }\prod_{\fk{l}\in\mathcal{L},\; \fk{l}\neq \fk{l}',\fk{l}'' }|v_{k_{\fk{l}}}| \lesssim N_{\fk{l}_1}^{-\frac{1}{2}}\|v\|_{l^2}^{2n_0-2}.
		\end{align}
		Note that on $\Lambda_{\fk{l}',\fk{l}''}$, $$|k_{\fk{l}'}|= |k_{\fk{l}''}|\sim |k_{\fk{n}_0}|\sim N_{\fk{l}'}\sim N_{\fk{l}''}\sim N_{\fk{l}_1}\sim N_{\fk{l}_2},\quad  N_{\fk{l}_3}\lesssim N_{\fk{l}_1}^{\theta}.$$
		We split the analysis in two cases.

		\noi
		$\bullet${\bf Case 1: $\iota_{\fk{n}_0}+\iota_{\fk{l}'}=0$}\\
		By (i) of Lemma \ref{boundPsi}, we have
		\begin{align*} 
			\mathrm{I}\lesssim &\underbrace{N_{\fk{l}_1}^{-2s}\sum_{\substack{\Lambda_{\fk{l}',\fk{l}''}\\ k_{\fk{n}_0}\neq k_{\fk{l}'} }}  \frac{N_{\fk{l}_1}^{2s-2}N_{\fk{l}_3}^2 }{\langle\Omega(\vec{k}_{\fk{r}})\rangle}\prod_{\fk{l}\in\mathcal{L}\setminus\{\fk{l}',\fk{l}' \} }|v_{k_{\fk{l}}}|}_{\mathrm{II}}+ \underbrace{N_{\fk{l}_1}^{-2s}N_{\fk{l}_3}^{2s}\sum_{\substack{\Lambda_{\fk{l}',\fk{l}''}\\ k_{\fk{n}_0}= k_{\fk{l}'} }} \prod_{\fk{l}\in\mathcal{L}\setminus\{\fk{l}',\fk{l}'' \} }|v_{k_{\fk{l}}}|}_{\mathrm{III}}.
		\end{align*}
		By Cauchy-Schwarz, 
		\begin{align}\label{III} 
			\mathrm{III}\lesssim N_{\fk{l}_1}^{-2s+2}N_{\fk{l}_3}^{2s+2}\Big(\prod_{j=4}^{2n_0-1}N_{\fk{l}_j}\Big)\prod_{\fk{l}\in\mathcal{L}\setminus\{\fk{l}',\fk{l}'' \} }\|v_{k_{\fk{l}}}\|_{l^2}\lesssim N_{\fk{l}_1}^{-2s+2+(2s+2n_0-4)\theta}\|v\|_{l^2}^{2n_0-2}.
		\end{align}
		To estimate II, we decompose $|\Omega(\vec{k}_{\fk{r}})|$ according to its level set and by Fubini, 
		\begin{align*}
			\mathrm{II}\lesssim & N_{\fk{l}_1}^{-2}N_{\fk{l}_3}^2\sum_{|\kappa|\lesssim N_{\fk{l}_1}N_{\fk{l}_3}}\frac{1}{\langle\kappa\rangle}\sum_{k_{\fk{l}},\fk{l}\neq\fk{l}',\fk{l}''}\prod_{\fk{l}\in\mathcal{L}\setminus\{\fk{l}',\fk{l}'' \} }|v_{k_{\fk{l}}}|\sum_{k_{\fk{n}_0},k_{\fk{l}'},k_{\fk{n}_0}\neq k_{\fk{l}'}}\mathbf{1}_{\Omega(\vec{k}_{\fk{r}})=\kappa}\mathbf{1}_{\Lambda_{\fk{l}',\fk{l}''} }.
		\end{align*}
		Note that for fixed $k_{\fk{l}},\fk{l}\in\mathcal{L}\setminus \{\fk{l}',\fk{l}'' \}$ and the resonant condition $\Omega(\vec{k}_{\fk{r}})=\kappa$,  the decoration$(k_{\fk{n}}:\fk{n}\in\mathcal{T}) \in \Lambda_{\fk{l}',\fk{l}''}$  implies that $|k_{\fk{n}_0}|^2-|k_{\fk{l}'}|^2=$const. and  $k_{\fk{n}_0}-k_{\fk{l}'}=$const. From the two-vector counting bound
		$$ \#\big\{(k_{\fk{n}_0},k_{\fk{l}'}):k_{\fk{n}_0}\neq k_{\fk{l}'}, |k_{\fk{n}_0}|^2-|k_{\fk{l}'}|^2=\mathrm{const.}, k_{\fk{n}_0}-k_{\fk{l}'}=\mathrm{const}.  \big\}\lesssim N_{\fk{l}_1},
		$$
		we obtain that
		\begin{align}\label{II} 
			\mathrm{II}\lesssim N_{\fk{l}_1}^{-1+}N_{\fk{l}_3}^2\prod_{j=3}^{2n_0}N_{\fk{l}_j}\prod_{\fk{l}\in\mathcal{L}\setminus\{\fk{l}',\fk{l}'' \} }\|v_{k_{\fk{l}}}\|_{l^2}\lesssim N_{\fk{l}_1}^{-1+2n_0\theta}\|v\|_{l^2}^{2n_0-2}.
		\end{align}
		Combining \eqref{III} and the fact that $\theta=\frac{1}{100m}<\frac{1}{4n_0}$, we get \eqref{dyadicIm}.

		\noi
		$\bullet${\bf Case 2: $\iota_{\fk{n}_0}=\fk{l}_{\fk{l}'}$}\\
		By (ii) of Lemma \ref{boundPsi}, we write
		\begin{align*}
			\mathrm{I}\lesssim N_{\fk{l}_1}^{-2s} \sum_{\Lambda_{\fk{l}',\fk{l}''}}N_{\fk{l}_1}^{2s-4}N_{\fk{l}_3}^2 \prod_{\fk{l}\in\mathcal{L}\setminus\{\fk{l}',\fk{l}'' \} }|v_{k_{\fk{l}}}|.
		\end{align*}
		By Cauchy-Schwarz, the left hand side of the above inequality is bounded by
		\begin{align*}
			N_{\fk{l}_1}^{-4}N_{\fk{l}_3}^2\sum_{k_{\fk{l}''},k_{\fk{l}},\fk{l}\neq \fk{l}', \fk{l}'' } \prod_{\fk{l}\in\mathcal{L}\setminus\{\fk{l}',\fk{l}'' \} }|v_{k_{\fk{l}}}|\lesssim N_{\fk{l}_1}^{-2}N_{\fk{l}_3}^2\prod_{j=3}^{2n_0}N_{\fk{l}_j}\prod_{\fk{l}\in\mathcal{L}\setminus\{\fk{l}',\fk{l}'' \} }\|v_{k_{\fk{l}}}\|_{l^2}.
		\end{align*}
		Thus
		\begin{align}\label{Case2I} 
			\mathrm{I}\lesssim N_{\fk{l}_1}^{-2+2n_0\theta}\|v\|_{l^2}^{2n_0-2}.
		\end{align}
		Since $\theta=\frac{1}{100m}=\frac{1}{50(n_0+1)}$, we get \eqref{dyadicIm}. This completes the proof of Proposition \ref{singularterm}.
	\end{proof}


	\section{Modified energy estimate II: Remainders }\label{nonsingular} 
	
	Given an expanded branching tree $\mathcal{T}$ of scale $\leq 2$ and of size $2n_0$ (with $n_0$ odd) with root $\fk{r}$ and the other branching node $\fk{n}_0$ (when the scale is $2$). Assume that $\# \mathcal{L}_{\fk{r}}=\#\mathcal{L}_{\fk{n}_0}=n_0$. Recall the definition $\mathcal{R}(\mathcal{T})$ in \eqref{RT}, by decomposing the components dyadically, we write
	\begin{align*}
		\mathcal{R}(\mathcal{T})(v):=\sum_{(N_{\fk{n}}:\fk{n}\in\mathcal{T}) } \mathcal{R}(\mathcal{T};(N_{\fk{n}})_{\fk{n}\in\mathcal{T}} )(v),
	\end{align*}
	where $\mathcal{R}(\mathcal{T};(N_{\fk{n}})_{\fk{n}\in\mathcal{T}} )(v)$ is the multilinear expression given by \eqref{RT}, adapted to the given dyadic scales $(N_{\fk{n}};\fk{n}\in\mathcal{T})$ (i.e.  $N_{\fk{n}}\leq |k_{\fk{n}}|< 2N_{\fk{n}}$). 
	Recall the quasi-order of leaves $\fk{l}_j$ so that   $N_{\fk{l}_1}\geq N_{\fk{l}_2}\geq\cdots N_{\fk{l}_{2n_0}}$. Note that we always have $N_{\fk{l}_1}\sim N_{\fk{l}_2}$. 
	
	The goal of this section is to prove:
	\begin{prop}\label{estimate:RT}
		Assume that $\delta<\frac{1}{100(n_0+s)}.$ There exist $C_0=C_0(s,n_0,\delta)>0,C_1=C_1(s,n_0)>0$ such that for any $p\geq 2$, $0<\epsilon <\delta^2, \lambda\geq 1$,
		$$ \|\mathcal{R}(\mathcal{T};(N_{\fk{n}})_{\fk{n}\in\mathcal{T}} )(v)\mathbf{1}_{\|v\|_{H^1}\leq \lambda} \|_{L^p(d\mu_s)}\lesssim_{\delta,\epsilon}p^{1-\epsilon}N_{\fk{l}_1}^{-C_1\delta+C_0\epsilon}\lambda^{2n_0},
		$$	
		where the implicit constant is independent of $p,\lambda,N_{\fk{n}},\fk{n}\in\mathcal{T}$.
	\end{prop}

	Consequently, combining Proposition \ref{singularterm} and summing over all adapted dyadic scales, we obtain that:
	\begin{cor}\label{finalestimatescale2}
		Let $\mathcal{T}$ be an expanded branching tree of scale $2$ and of size $2n_0$, with root $\fk{r}$ and the other branching node $\fk{n}_0$ (when the scale is $2$) and  $\# \mathcal{L}_{\fk{r}}=\#\mathcal{L}_{\fk{n}_0}=n_0$. Then for any $p\geq 2$, $N\in \N\cup\{\infty\}$ and any sufficiently small $\epsilon>0$, we have
		$$ \|\mathcal{T}_{e}(\pi_Nv)\mathbf{1}_{\|v\|_{H^1}\leq \lambda} \|_{L^p(d\mu_s)}\lesssim_{\epsilon}p^{1-\epsilon}\lambda^{2n_0},
		$$
		where the tree expression $\mathcal{T}_e(\cdot)$ is given by \eqref{Te}.
	\end{cor} 
	
	In order to prove Proposition \ref{estimate:RT}, we need to establish additional estimates to treat the situation where there are no pairing among $\delta$-dominated leaves in $\mathcal{L}_{\delta}^*$ (defined in \eqref{dominatedleaves}).

	\subsection{Estimates for unpaired dominated leaves}
	Given an expanded branching tree of scale $2$ and of size $2n_0$, with root $\fk{r}$ and non-root node $\fk{n}_0$, adapted to dyadic numbers $(N_{\fk{n}}:\fk{n}\in\mathcal{T})$, with the non-increasing arrangement $N_{\fk{l}_1}\geq\cdots N_{\fk{l}_{2n_0}}$. Fix $0<\delta<\frac{1}{100(n_0+s)}$, denote by $(k_{\fk{n}}:\fk{n}\in\mathcal{T})_{\emptyset}$ a decoration such that there is no pairing among $\delta$-dominated leaves $\mathcal{L}_{\delta}^*$, defined in \eqref{dominatedleaves}. The associated tree expression is 
	\begin{align}\label{tree-nopairing}
		\mathcal{T}_{e,\emptyset}(a^{(\fk{l})};\fk{l}\in\mathcal{L}):=\sum_{(k_{\fk{n}}:\fk{n}\in\mathcal{T})_{\emptyset}}\frac{\psi_{2s}(\vec{k}_{\fk{r}})}{\Omega(\vec{k}_{\fk{r}})}\prod_{\fk{l}\in\mathcal{L}}(a^{(\fk{l})}_{k_{\fk{l}}})^{\iota_{\fk{l}}}
	\end{align}
	Note that the tree expression $\mathcal{T}_{e,\emptyset}$ implicitly depends on the parameter $\delta$. 
	\begin{lem}\label{nopairingl=3} 
		Assume that $l=\#\mathcal{L}_{\delta}^*\geq 3$ and $a^{(\fk{l})}_{k_{\fk{l}}}=\mathbf{1}_{N_{\fk{l}}\leq |k_{\fk{l}}|<2N_{\fk{l}} }\widehat{\phi}(k_{\fk{l}})$. Then for any $\epsilon>0$,
		$$ \|\mathcal{T}_{e,\emptyset}(a^{(\fk{l})};\fk{l}\in\mathcal{L})\mathbf{1}_{\|\phi\|_{H^1}\leq \lambda} \|_{L^p(d\mu_s)}\lesssim_{\delta,\epsilon} p^{1-\epsilon}N_{\fk{l}_1}^{-\frac{l-2}{l}+\frac{2\delta}{l}((l-2)(s-1)-1)+C\epsilon},
		$$ 
		where $C=C(s,l,\delta)>0$.  
	\end{lem}

	\begin{lem}\label{nopairinglg=2}
		Assume that $l=\#\mathcal{L}_{\delta}^*=2$ and $a^{(\fk{l})}_{k_{\fk{l}}}=\mathbf{1}_{N_{\fk{l}}\leq |k_{\fk{l}}|<2N_{\fk{l}} }\widehat{\phi}(k_{\fk{l}})$. Then for any $\epsilon>0$, we have
		$$ \|\mathcal{T}_{e,\emptyset}(a^{(\fk{l})};\fk{l}\in\mathcal{L})\mathbf{1}_{\|\phi\|_{H^1}\leq \lambda} \|_{L^p(d\mu_s)}\lesssim_{\delta,\epsilon} p^{1-\epsilon}N_{\fk{l}_1}^{-\delta(1-\epsilon)+C\epsilon},
		$$ 
		where $C=C(s,l)>0$.
	\end{lem}
It is important to note that there is no assumption $\#\mathcal{L}_{\fk{r}}=\#\mathcal{L}_{\fk{n}_0}$ in Lemma \ref{nopairingl=3} and Lemma \ref{nopairinglg=2}.

	Following the same proof of Lemma \ref{nopairingl=3} and Lemma \ref{nopairinglg=2}, we also have :
			\begin{lem}\label{scale1} 
				Let $\mathcal{T}$ be an expanded branching tree of scale $1$ and size $2n_0$, with root $\fk{r}$, adapted to dyadic numbers $(N_{\fk{n}}:\fk{n}\in\mathcal{T})$. Then there exist $C_0=C_0(s,n_0,\delta)>0$, $C_1=C_1(s,n_0)>0$ such that for any $p\geq 2$, $\epsilon>0$, we have
				$$ \|\mathcal{T}_{e,\emptyset}(a^{(\fk{l})};\fk{l}\in\mathcal{L})\mathbf{1}_{\|\phi\|_{H^1}\leq \lambda} \|_{L^p(d\mu_s)}\lesssim_{\delta,\epsilon} p^{1-\epsilon}N_{\fk{l}_1}^{-\delta(1-\epsilon)+C\epsilon},
				$$ 
				where  the implicit constant does not depend on $p, \lambda$ and $N_{\fk{n}},\fk{n}\in\mathcal{T}$.
			\end{lem} 

			Consequently, summing over dyadic sizes, we have:
			\begin{cor}\label{finalestimatescale1}
				Let $\mathcal{T}$ be an expanded branching tree of scale $1$, size $2n_0$ and with root $\fk{r}$. 
				Consider the tree expression $\widetilde{\mathcal{T}}_e(\cdot)$ is given by 
				$$\widetilde{\mathcal{T}}_e(v):=\sum_{(k_{\fk{n}}:\fk{n}\in\mathcal{T}) } \frac{\psi_{2s}(\vec{k}_{\fk{r}})}{\Omega(\vec{k}_{\fk{r}})}\prod_{\fk{l}\in\mathcal{L}}v_{k_{\fk{l}}}^{\iota_{\fk{l}}},
				$$
				with the convention that $\frac{\psi_{2s}(\vec{k}_{\fk{r}})}{\Omega(\vec{k}_{\fk{r}})}=\psi_{2s}(\vec{k}_{\fk{r}})$ when $\Omega(\vec{k}_{\fk{r}})=0$.  
				Then there exists as sufficiently small $\epsilon_0>0$, such that for any $p\geq 2$, $N\in \N\cup\{\infty\}$, we have
				$$ \|\widetilde{\mathcal{T}}_{e}(\pi_Nv)\mathbf{1}_{\|v\|_{H^1}\leq \lambda} \|_{L^p(d\mu_s)}\lesssim_{\epsilon_0}p^{1-\epsilon_0}\lambda^{2n_0}.
				$$
			\end{cor} 

	We postpone the proof of Lemma \ref{nopairingl=3}, Lemma \ref{nopairinglg=2} and use them to prove Proposition \ref{estimate:RT} and finish the proof of the main energy estimates.
	
	\subsection{Proof of the Proposition \ref{estimate:RT} and the modified energy estimates}

	\begin{proof}[Proof of Proposition \ref{estimate:RT}]
		
		Given dyadic sizes of frequencies $(N_{\fk{n}}:\fk{n}\in\mathcal{T})$, we denote by $N_{\fk{l}_1}\geq\cdots\geq  N_{\fk{l}_l}\geq N_{\fk{l}_1}^{1-\delta}$ the $\delta$-dominated frequency scales, associated with leaves $\fk{l}_1,\cdots,\fk{l}_l$, $l\geq 2$.  
		
		We decompose
		$\mathcal{R}(\mathcal{T};(N_{\fk{n}})_{\fk{n}\in\mathcal{T}} )(v)
		$ as
		$$ \sum_{(X,Y,\sigma;\mathcal{L}_{\delta}^{*}) }
		\mathcal{T}_{e,X,Y}^{\sigma}
		(\widehat{v}_{k_{\fk{l}}};\fk{l}\in\mathcal{L} ),
		$$
		where we recall that $(X,Y,\sigma,\mathcal{L}_{\delta}^*)$ is a pairing from $X$ to $Y$ with $X,Y\subset \mathcal{L}_{\delta}^*$ and \eqref{crosspairedtree} the associated tree expression. If $X=Y=\emptyset$, we adopt the convention that $\mathcal{T}_{e,X,Y}^{\sigma}=\mathcal{T}_{e,\emptyset}$ if $X=Y=\emptyset$. 
		
		Note that from the definition  \eqref{RT}, in the above expression, if $l=2$ and $N_{\fk{l}_3}\ll N_{\fk{l}_1}^{\theta}$, $\theta=\frac{1}{50(n_0+1)}$ the paired sets of leaves $X,Y$ cannot belong to two different generations.  
		
		Given $X,Y,\sigma$, we execute following operations for the tree $\mathcal{T}$ when estimating $$\big\|\mathcal{T}_{e,X,Y}^{\sigma} (\widehat{v}_{k_{\fk{l}}};\fk{l}\in\mathcal{L})\mathbf{1}_{\|\phi\|_{H^1}\leq \lambda }
		\big\|_{L^p(d\mu_s)}:$$
		\begin{itemize}
			
			\item[(a)] If $X\cap \mathcal{L}_{\delta}^*\neq \emptyset$, go to the operation (b).
			Otherwise, $X\cap \mathcal{L}_{\delta}^*=\emptyset$.
			we apply Lemma \ref{nopairinglg=2} when $\#\mathcal{L}_{\delta}^*=2$ and Lemma \ref{nopairingl=3} when $\#\mathcal{L}_{\delta}^*\geq 3$. Then stop.
			
			\item[(b)] If there exists $\fk{l}\in X\cap \mathcal{L}_{\delta}^* $ such that $\fk{l},\sigma(\fk{l})$ belong to the same generation, go to the operation (c). Otherwise,  the leaves $\fk{l}$ and $\sigma(\fk{l})$ belong to two different generations, then apply Proposition \ref{crosspaired} and stop.
			
			\item[(c)]
			Cut a couple of paired leaves $(\fk{l},\sigma(\fk{l})),\fk{l}\in X\cap\mathcal{L}^*$ such that $\fk{l},\sigma(\fk{l})$ are in the same generation. Renew the tree $\mathcal{T}$ with the new set of leaves $\mathcal{L}\setminus\{\fk{l},\sigma(\fk{l}) \}$, renew the set of $\delta$-dominated leaves $\mathcal{L}_{\delta}^{*}$ and the new set of paired leaves $X,Y$. 
			Then go back to the operation (a). 
		\end{itemize}
		Note that $\# X\leq 2m-1$ is finite, the above algorithm will stop at operations (a) or (b), in finitely many steps. It suffices to check the output. If the algorithm stops at the operation (a), we obtain the final bound $C_{\epsilon}p^{1-\epsilon}\lambda^{2n_0}N_{\fk{l}_1}^{-\frac{1}{3}+C_0\epsilon+C_1\delta},$ which is conclusive. If the algorithm stops at the operation (b), we distinguish two cases. If the algorithm stops at (b) in the first cycle, then we must have $N_{\fk{l}_3}\gtrsim N_{\fk{l}_1}^{\theta}$, so the output of Proposition \ref{crosspaired} can be bounded by
		$ C_{\epsilon}p^{1-\epsilon}\lambda^{2n_0}(N_{\fk{l}_1}^{-\delta}+N_{\fk{l}_1}^{-\theta}),
		$
		which is conclusive. If else, the algorithm stops at the operation (b) through more than one cycle,  we will gain $N_{\fk{l}_1}^{-2(s-1)}$ through the operation (c). Viewing that Proposition \ref{crosspaired} losses at most $N_{\fk{l}_1}^{\epsilon}$, in the final output we still gain $N_{\fk{l}_1}^{-2(s-1)+\epsilon}$, which is conclusive.
		
		To summary, the proof of Proposition \ref{estimate:RT} is complete.
		
	\end{proof} 
	
	Next we finish the proof of the main results in Section \ref{Sec:modifiedenergy}.
	
	\begin{proof}[Proof of Proposition \ref{weightedmeasure}]
				Recall the definition of the functional $R_{s}$ in \eqref{Rst} and its truncated version $R_{s,N}$. Using the tree notation, we rewrite $R_{s,N}(v)=\widetilde{\mathcal{T}}_e(\pi_Nv)$, parametrized by an expanded branching tree of scale $1$, size $2m$ and with root $\fk{r}$. By Corollary \ref{finalestimatescale1}, we obtain that
				$$ \big\|R_{s,N}(u)\mathbf{1}_{H_N[u]\leq \lambda}\big\|_{L^p(d\mu_s)}\leq C_{\epsilon_0}p^{1-\epsilon_0}\lambda^{2m}.
				$$  
			This implies that for any $1\leq r<\infty$, $$\big\|\mathbf{1}_{H_N[u]\leq \lambda } \cdot \e^{|R_{s,N}(u)| }  \big\|_{L^r(d\mu_s)}\leq C(\epsilon_0,r,\lambda).
				$$
		In order prove the last assertion in Proposition \ref{weightedmeasure}, it suffices to show that $R_{s,N}[u]\mathbf{1}_{H_N[u]\leq \lambda}\rightarrow R_{s}[u]\mathbf{1}_{H[u]\leq \lambda}$ in measure $\mu_s$, as $N\rightarrow\infty$. 
				
				Indeed, since $\mathbf{1}_{H_N[u]\leq\lambda}\rightarrow \mathbf{1}_{H[u]\leq\lambda}$ $\mu_s$-a.e., it suffices to show $R_{s,N}[u]\mathbf{1}_{H[u]\leq \lambda}\rightarrow R_{s}[u]\mathbf{1}_{H[u]\leq \lambda}$. This is a consequence of Lemma \ref{scale1}. 
				To see this, we remark that
				$R_{s,N}[u]\mathbf{1}_{H[u]\leq \lambda}-R_{s}[u]\mathbf{1}_{H[u]\leq\lambda }
				$
				consists of multi-linear tree expression (only of scale $1$) with components $\pi_Nu,u,\pi_N^{\perp}u$ and there is at least one $\pi_N^{\perp}u$.  By decomposing into dyadic sizes and cutting paired leaves, we can apply Lemma \ref{scale1} to each single term and sum over dyadic sizes. Since there must be at least one component of $\pi_N^{\perp}u$, we must have $N_{\fk{l}_1}>N$. This means that we only need to sum over dyadic number $N_{\fk{l}_1}> N$. Hence the convergence of  $R_{s,N}[u]\mathbf{1}_{H[u]\leq \lambda}-R_{s}[u]\mathbf{1}_{H[u]\leq\lambda }$ in $L^p(d\mu_s)$ follows.    
	\end{proof}
	
	\begin{proof}[Proof of Proposition \ref{energyestimate}]
			Note that $Q_{N}(u)$ can be decomposed into finitely many tree expansion $\Im\mathcal{T}_e(u)$ for expanded branching trees of scale $2$ and size $2m$. For each fixed tree $\mathcal{T}$, we adapt the decomposition \eqref{RT}, using Proposition \ref{singularterm} to estimate the imaginary part of the singular contribution and Proposition \ref{estimate:RT} to estimate the remainder. Note that in applications of these estimates, we fix the parameters $\delta,\epsilon$ such that $\delta<\frac{1}{100(n_0+s)}$ and $\epsilon=\epsilon_0\ll \delta$ such that the exponent of $N_{\fk{l}_1}$ is strictly negative. 
			 This completes the proof of Proposition \ref{energyestimate}.
	\end{proof}


	\subsection{Proof of the key estimates}
	
	\begin{proof}[Proof of Lemma \ref{nopairingl=3}] 
	Without loss of generality, we denote the non-root node $\fk{n}_0=\fk{r}_1$.  Since cutting paired leaves within the same generation will not change the structure of the tree, we further assume that there is no pairing within $\mathcal{L}_{\fk{r}}$ or $\mathcal{L}_{\fk{n}_0}$.

	Recall that $l=\#\mathcal{L}_{\delta}^*$, the first $l$-Gaussians $a^{(\fk{l}_1)},\cdots a^{(\fk{l}_{l})}$ are independent of $a^{(\fk{l}_j)}, j\geq l+1$ (if $l<2n_0$),
	the conditional Wiener chaos estimate (using the non-pairing condition for $k_{\fk{l}_1},\cdots,k_{\fk{l}_l}$) yields
	\begin{align}\label{Wiener}  \|\mathcal{T}_{e,\emptyset}(a^{(\fk{l})};\fk{l}\in\mathcal{L})\mathbf{1}_{\|\phi\|_{H^1}\leq \lambda} \|_{L^p(d\mu_s)}\lesssim p^{\frac{l}{2}}\Big(\prod_{j=1}^lN_{\fk{l}_j}^{-s}\Big)\Big(\sum_{k_{\fk{l}_i}:1\leq i\leq l}|\mathcal{U}(k_{\fk{l}_1},\cdots,k_{\fk{l}_l})|^2
		\Big)^{\frac{1}{2}},
	\end{align}
	where
	$$ \mathcal{U}(k_{\fk{l}_1},\cdots,k_{\fk{l}_l})=\sum_{k_{\fk{l}_j}: \; l< j\leq 2n_0 }\mathbf{1}_{\sum_{\fk{l}}\iota_{\fk{l}}k_{\fk{l}}=0}\frac{\psi_{2s}(\vec{k}_{\fk{r}})}{\Omega(\vec{k}_{\fk{r}})}\prod_{j=l+1}^{2n_0}(a_{k_{\fk{l}_j}}^{(\fk{l}_j)})^{\iota_{\fk{l}_j}},
	$$
	and in the above sums, we implicitly insert the cutoffs $N_{\fk{l}_j}\leq |k_{\fk{l}_j}|<2N_{\fk{l}_j}$.

	By Cauchy-Schwarz, we have
	\begin{align}
		&\sum_{k_{\fk{l}_i}:1\leq i\leq l } 
		|\mathcal{U}(k_{\fk{l}_1},\cdots,k_{\fk{l}_l})|^2
		\notag\\
		\lesssim &\Big(\prod_{j=l+1}^{2n_0}\|a_{k_{\fk{l}_j}}^{(\fk{l}_j)}\|_{l^2}^2\Big)\sum_{k_{\fk{l}_j}; 1\leq j\leq l}\Big(\sum_{k_{\fk{l}_i},l< i\leq 2n_0 }\Big|\frac{\psi_{2s}(\vec{k}_{\fk{r}})}{\Omega(\vec{k}_{\fk{r}}) } \Big|^2 \mathbf{1}_{\sum_{j=1}^{2n_0}\iota_{\fk{l}_j}k_{\fk{l}_j}=0 }\Big) \label{presumck}
		\\ \lesssim &\Big(\prod_{j=l+1}^{2n_0}\lambda^2N_{\fk{l}_j}^{-2} \Big)
		\sum_{(k_{\fk{n}}:\fk{n}\in\mathcal{T})} \Big|\frac{\psi_{2s}(\vec{k}_{\fk{r}})}{\Omega(\vec{k}_{\fk{r}}) } \Big|^2.
		\label{sumck} 
	\end{align}
	Applying Lemma \ref{lem:wcounting}, we deduce that
	\begin{align}\label{output3leaves} 
		\sum_{k_{\fk{l}_i}:1\leq i\leq l }|\mathcal{U}(k_{\fk{l}_1},\cdots,k_{\fk{l}_l})|^2\lesssim_{\epsilon} \lambda^{2(2n_0-l)}\cdot N_{\fk{l}_1}^{4s-2+\epsilon}N_{\fk{l}_3}^4\prod_{j=4}^lN_{\fk{l}_j}^2,
	\end{align}
	where we adopt the convention that $\prod_{j=4}^lN_{\fk{l}_j}^2:=1$ if $l=3$. Plugin, we have
	\begin{align*} 
		\|\mathcal{T}_{e,\emptyset}(a^{(\fk{l})};\fk{l}\in\mathcal{L})\mathbf{1}_{\|\phi\|_{H^1}\leq \lambda} \|_{L^p(d\mu_s)}\lesssim &p^{\frac{l}{2}}\lambda^{2n_0-l}N_{\fk{l}_1}^{-1+}N_{\fk{l}_3}^{-(s-2)}\prod_{j=4}^lN_{\fk{l}_j}^{-(s-1)}\\
		\leq & p^{\frac{l}{2}}\lambda^{2n_0-l}N_{\fk{l}_1}^{-\big((l-2)(s-1)(1-\delta) +\delta\big)+}.
	\end{align*} 
	thanks to the fact that $N_{\fk{l}_j}>N_{\fk{l}_1}^{1-\delta}$ for $1\leq j \leq l$.

	On the other hand, by Lemma \ref{Basictreeestimate} (\eqref{eq:dyadictreeestimate'}), we have
	$$ \|\mathcal{T}_{e,\emptyset}(a^{(\fk{l})};\fk{l}\in\mathcal{L})\mathbf{1}_{\|\phi\|_{H^1}\leq \lambda} \|_{L^p(d\mu_s)}\lesssim_{\epsilon} \lambda^{2n_0}N_{\fk{l}_1}^{2(s-2)+\epsilon}N_{\fk{l}_3}.
	$$
	Interpolating these two inequalities, we complete the proof of Lemma \ref{nopairingl=3}.
	
\end{proof}

\begin{proof}[Proof of Lemma \ref{nopairinglg=2}]
	
Following the same argument as in the proof of Lemma \ref{nopairingl=3}, we obtain the estimate

	\begin{align*} 
	\|\mathcal{T}_{e,\emptyset}(a^{(\fk{l})};\fk{l}\in\mathcal{L})\mathbf{1}_{\|\phi\|_{H^1}\leq \lambda} \|_{L^p(d\mu_s)}\lesssim_{\epsilon}\hspace{0.1cm} &p\lambda^{2n_0-2}N_{\fk{l}_1}^{-2s}\cdot \prod_{j=3}^{2n_0}N_{\fk{l}_j}^{-1}\cdot N_{\fk{l}_1}^{(2s-1)+\epsilon}N_{\fk{l}_3}^2\prod_{j=4}^{2n_0}N_{\fk{l}_j}\\
	=& p\lambda^{2n_0-2}N_{\fk{l}_1}^{-1+\epsilon}N_{\fk{l}_3}\leq p\lambda^{2n_0-2}N_{\fk{l}_1}^{-\delta+\epsilon},
\end{align*} 
since by definition $l=\#\mathcal{L}_{\delta}^*=2$, we have $N_{\fk{l}_3}\leq N_{\fk{l}_1}^{1-\delta}$. Interpolating with any trivial deterministic estimate, we complete the proof of Lemma \ref{nopairinglg=2}.
\end{proof}	


	\section{On the global nonlinear smoothing effect}
	
\subsection{The global nonlinear smoothing effect }
Let us start with a local nonlinear smoothing result, essentially due to Bourgain \cite{Bo96}, see also \cite{CLSta} or \cite{OWang} for a pedagogical proof :
\begin{prop}\label{localstructure}
	Let $s>2$, $\sigma<s-1, s_1<s-\frac{1}{2}$. For any $R\geq 1$, there exists a Borel measurable set $\Sigma_R\subset H^{\sigma}(\T^2)$, such that for any $\phi\in \Sigma_R$, 
	the solution $u(t)$ of the cubic NLS \eqref{NLS} with initial data $\phi$ on $t\in[-R^{-\kappa_0},R^{-\kappa_0}]$, with some absolute constant $\kappa_0>0$,
satisfies 
	$$ \sup_{|t|\leq R^{-\kappa_0}}\Big\|
	\e^{\frac{it}{2\pi^2}\int_{\T^2}|\phi^{\omega}|^2dx}\cdot u(t)- \e^{it\Delta}\phi^{\omega}
	\Big\|_{H^{s_1}(\T^2)}\leq 1.
	$$
	Moreover, the complement of the set $\Sigma_R$ has $\mu_s$-measure less than $C\e^{-cR^{\alpha}}$, for some absolute constants $C,c>0$ and $ \alpha\in (0,1)$.
\end{prop}
Now we are able to extend Proposition \ref{localstructure} to any time interval $[-T,T]$, thus proving Theorem \ref{globalstructure}.
	\begin{proof}[Proof of Theorem \ref{globalstructure}]
		By Theorem \ref{thm:main},
		$$ \mathrm{supp}(\mu_s)=\bigcup_{\lambda=1}^{\infty}\mathrm{supp}(\mu_{s,\lambda}).
		$$
		It suffices to show that, for any fixed $\lambda>0$, $\mu_{s,\lambda}$-almost surely, the decomposition  \eqref{recenter2} holds. 
		In the sequel, all the constants $C$ can implicitly depend on $\lambda$, without additional declaration.
		
		Fix any $T\ge 1$. For $R\geq 1$, denote by $\Sigma_R$, the data set in Theorem \ref{localstructure}. 
		By Proposition \ref{weightedmeasure}, we have $\rho_{s,\lambda}(( \Sigma_{R})^c)<C\e^{-cR^{\alpha}}$.
		Set $\tau=R^{-\kappa_0}$, define $$\Sigma_{R,k}:=\Phi^{-1}_{k\tau}(\Sigma_{R}),\;\text{ for } |k|\leq K_0:=\frac{T}{\lfloor \tau\rfloor}+1\sim TR^{\kappa_0}. $$
		Let
		$$ \widetilde{\Sigma}_{R,T}:=\bigcap_{k=-K_0}^{K_0}\Sigma_{R,k}.
		$$
		We have
		\begin{align}
			\rho_{s,\lambda}((\widetilde{\Sigma}_{R,T})^c)\leq &\sum_{|k|\leq K_0}\rho_{s,\lambda}((\Sigma_{R,k})^c)
			=\sum_{|k|\leq K_0}\rho_{s,\lambda}(\Phi_{k\tau}^{-1}((\Sigma_{R,0})^c)) \notag\\
			\leq &\sum_{|k|\leq K_0}\|f_{\lambda}(k\tau,\cdot)\|_{L^2(d\rho_{s,\lambda})}^{\frac{1}{2}}\big(\rho_{s,\lambda}((\Sigma_{R})^c)\big)^{\frac{1}{2}},\notag 
		\end{align}
		where we used the definition of the Radon-Nikodym density and Cauchy-Schwarz in the last step. 
	   
	   Now from the transported density bound
		$$ 
		\sup_{|t|\leq T}\|f_{\lambda}(t,\cdot)\|_{L^p(d\rho_{s,0,\lambda})}\leq C\e^{1+T^{C_0}},
		$$
		we obtain that
		\begin{align}\label{bdtildeSigma}  \rho_{s,\lambda}(( \widetilde{\Sigma}_{R,T} )^c )\leq C\e^{C(1+T^{C_0})}\cdot TR^{\kappa_0}\e^{-cR^{\alpha}}.
		\end{align}
		Let
		$$ \widetilde{\Sigma}_T:=\bigcup_{R_1=1}^{\infty}\bigcap_{R=R_1}\widetilde{\Sigma}_{R,T}.
		$$
		Thanks to \eqref{bdtildeSigma} and the Borel-Cantelli Lemma, we deduce that $\widetilde{\Sigma}_T$ has the full $\rho_{s,\lambda}$ measure. 
		
		To complete the proof, it remains to show that for any $\phi\in \widetilde{\Sigma}_T$, it admits the decomposition \eqref{recenter2}. This follows from the superposition of the re-centering structure of each small interval of the local theory.
		
		More precisely,  by definition of $\widetilde{\Sigma}_T$, there exists $R_1\geq 1$ such that $\phi\in \widetilde{\Sigma}_{R,T}$ for all $R\geq R_1$. Let us fix some $R\geq R_1$ and denote by $I_k=[(k-1)\tau,k\tau]$, $\tau=R^{-\kappa_0}$, for $1\leq k\leq K_0=T/\lfloor\tau \rfloor+1$. Since $\Phi_{(k-1)\tau}\phi\in \Sigma_{R}$, for $t\in I_k$, by the flow property, we have the decomposition
		$$ \Phi_t\phi=\Phi_{t-k\tau+\tau}\circ\Phi_{(k-1)\tau}\phi=\e^{-\frac{it}{2\pi^2}\|\phi\|_{L^2}^2}\cdot\big(\e^{i(t-(k-1)\tau)\Delta}(\Phi_{(k-1)\tau}\phi)+ w_k(t)\big),\;\text{ for }t\in I_k
		$$
		with $\|w_k(t)\|_{H^{s_1}}\leq 1$,
		where we have used the mass conservation $\|\Phi_t\phi\|_{L^2}=\|\phi\|_{L^2}$. Iterating $K_0$ steps, we obtain \eqref{recenter2} with an estimate
		$$ \sup_{0\leq t\leq T}\|w(t)\|_{H^{s_1}(\T)}\leq CR^{\kappa_0}T,
		$$
		since we add up $O(R^{\kappa_0}T)$ terms in the space $H^{s_1}$ with norm $O(1)$. 
		The proof of Theorem \ref{globalstructure} is now complete.
	\end{proof}

\subsection{Lack of regularization of the first Picard's iteration }
	
\begin{proof}[Proposition of \ref{prop:lacksmoothing}]  
	For a given initial datum $\phi$, let $\mathcal{I}_1(\phi;t)$ denote the first Picard iterate of the Duhamel term. A direct computation in Fourier variables gives
	\begin{align}\label{eq:FirstPicardPatch}
		\widehat{\mathcal{I}_1(t;\phi)}_k
		&= \sum_{\substack{k_1-k_2+k_3=k\\ k_2\neq k_1,k_3}}
		\frac{2\sin\!\big(\tfrac{t\,\Omega(\vec{k})}{2}\big)}{i\,\Omega(\vec{k})}\,e^{-i\frac{t\,\Omega(\vec{k})}{2}}
		\,\widehat{\phi}(k_1)\,\overline{\widehat{\phi}}(k_2)\,\widehat{\phi}(k_3)
		\;+\;\frac{t}{i}\,|\widehat{\phi}(k)|^2\,\widehat{\phi}(k),
	\end{align}
	where $\Omega(\vec{k}):=|k_1|^2-|k_2|^2+|k_3|^2-|k|^2$.
	
	Fix $n\ge1$ and define $\widehat{\phi_n}$ as a sum of one thin vertical slab and two atoms:
	\[
	\widehat{\phi_n}(\xi,\eta)\;=\;f_n(\xi,\eta)+a_1(\xi,\eta)+a_2(\xi,\eta),
	\]
	with
	\[
	f_n(\xi,\eta):=n^{-\sigma-\frac12}\,\mathbf{1}_{\{\xi=0,\ \eta\in[n,2n]\}},\qquad
	a_1(\xi,\eta):=\mathbf{1}_{\{\xi=1,\ \eta=0\}},\qquad
	a_2(\xi,\eta):=\mathbf{1}_{\{\xi=2,\ \eta=0\}}.
	\]
	Then $\|\phi_n\|_{H^\sigma}\sim 1$ uniformly in $n$, since
	\[
	\sum_{\eta=n}^{2n}\langle(0,\eta)\rangle^{2\sigma}\,n^{-2\sigma-1}\;\sim\;1,
	\qquad
	\|(a_1+a_2)\|_{H^\sigma}\lesssim 1.
	\]
	
	We now evaluate \eqref{eq:FirstPicardPatch} at frequencies $k=(1,\eta)$ with $\eta\in[n,2n]$. 
	Among all triples $(k_1,k_2,k_3)$ with $k_1-k_2+k_3=k$, the only ones that produce a nonzero contribution with the above data and land at $(1,\eta)$ are the two placements
	\[
	\big(k_1,k_2,k_3\big)=\big((0,\eta),(1,0),(2,0)\big)\quad\text{and}\quad
	\big(k_1,k_2,k_3\big)=\big((2,0),(1,0),(0,\eta)\big).
	\]
	In both cases,
	\[
	\Omega(\vec{k})=|k_1|^2-|k_2|^2+|k_3|^2-|k|^2=2,
	\]
	so the time prefactor equals
	\[
	\frac{2\sin\!\big(\frac{t\Omega}{2}\big)}{i\,\Omega}\,e^{-i\frac{t\Omega}{2}}
	=\frac{2\sin t}{2i}\,e^{-it}
	=\frac{\sin t}{i}\,e^{-it}.
	\]
	Moreover, the gauge term $\frac{t}{i}|\widehat{\phi}_k|^2\widehat{\phi}_k$ vanishes at $k=(1,\eta)$ for $\eta\in[n,2n]$, since $\widehat{\phi_n}(1,\eta)=0$ there. 
	Therefore, for all $\eta\in[n,2n]$,
	\[
	\widehat{\mathcal{I}_1(t;\phi_n)}_{(1,\eta)}
	= C\,\sin t\,e^{-it}\;n^{-\sigma-\frac12},
	\]
	where $C\neq 0$ is an absolute constant (coming from the two contributing triples). In particular,
	\[
	\big|\widehat{\mathcal{I}_1(t;\phi_n)}_{(1,\eta)}\big|
	\gtrsim |\sin t|\;n^{-\sigma-\frac12},\qquad \eta\in[n,2n].
	\]
	
	Consequently, using $\langle(1,\eta)\rangle^{2\sigma_1}\sim \eta^{2\sigma_1}\sim n^{2\sigma_1}$ on the block $\eta\in[n,2n]$, we obtain
	\[
	\|\mathcal{I}_1(t;\phi_n)\|_{H^{\sigma_1}}^2
	\ \ge\ \sum_{\eta=n}^{2n}\langle(1,\eta)\rangle^{2\sigma_1}\,
	\big|\widehat{\mathcal{I}_1(t;\phi_n)}_{(1,\eta)}\big|^2
	\ \gtrsim\ n\cdot n^{2\sigma_1}\cdot\big(|\sin t|^2\,n^{-2\sigma-1}\big)
	\ =\ |\sin t|^2\,n^{2(\sigma_1-\sigma)}.
	\]
	Hence, for every $t\notin E:=\pi\Z$, we have $\|\mathcal{I}_1(t;\phi_n)\|_{H^{\sigma_1}}\to\infty$ as $n\to\infty$ whenever $\sigma_1>\sigma$. This completes the proof of Proposition \ref{prop:lacksmoothing}.
\end{proof}


	
	\appendix


	\section{$L^p$ bound for the Radon--Nikodym density}
	
	In this appendix, we state an abstract lemma to link a differential inequality of a family of equivalent transported measures to the $L^p$ bound of the corresponding transported density. 
	\begin{lem}\label{abstract}
		Let $X$ be a metric space and let $(\mu_t)_{t\in\mathbb{R}}$ be a family of finite Borel measures on $X$ that are pairwise equivalent. Assume that for every $p\in(1,\infty)$ there exists $M_p>0$ such that for all $t_1,t_2\in\mathbb{R}$,
		\begin{equation}\label{changeweight}
			\Big\|\frac{d\mu_{t_1}}{d\mu_{t_2}}\Big\|_{L^p(d\mu_{t_2})}\le M_p.
		\end{equation}
		
		Let $\Phi_t:X\to X$ be a continuous flow and assume that, for every Borel set $A\subset X$, the map $t\mapsto \mu_t(\Phi_t^{-1}(A))$ is differentiable. Moreover, assume that there exist $C_0>0$ and $\alpha\in(0,1)$ such that for any Borel set $A\subset X$, any $r\in(1,\infty)$, and any $t\in\mathbb{R}$,
		\begin{equation}\label{differential}
			\frac{d}{dt}\,\mu_t(\Phi_t^{-1}(A))\ \le\ C_0\, r^{1-\alpha}\,\mu_t(\Phi_t^{-1}(A))^{\,1-\frac1r}.
		\end{equation}
		Then for any $\delta\in(0,1)$ and $T\ge 1$, there exists $C_{\delta,T}>0$ such that, uniformly for $|t|\le T$ and all Borel $A$,
		\begin{equation}\label{recursive}
			\mu_t(\Phi_t^{-1}(A))\ \le\ C_{\delta,T}\,\mu_t(A)^{\,1-\delta}.
		\end{equation}
		Furthermore, for the Radon--Nikodym density
		\[
		f_t:=\frac{d\big((\Phi_t)_*\mu_t\big)}{d\mu_t},
		\]
		we have, for every $p\in(1,\infty)$,
		\begin{equation}\label{Lp:ftu}
			\sup_{|t|\le T}\|f_t\|_{L^p(d\mu_t)}\ \le\ C(\alpha,p)\,\exp\!\Big(C(\alpha,p)\big(1+T^{2/\alpha}\big)\Big).
		\end{equation}
	\end{lem}
	\begin{proof}
		The argument below is essentially the same appeared in \cite{GLTz}. 
		By definition, for any Borel set $A\subset X$,
		\[
		\mu_t(\Phi_t^{-1}(A))=\int_A f_t\,d\mu_t.
		\]
		Thus \eqref{recursive} follows from \eqref{Lp:ftu} and Hölder. It suffices to prove the weak $L^q$ bound, uniform in $|t|\le T$, for any fixed $q\in(1,\infty)$:
		\begin{equation}\label{weakLqbound}
			\sup_{|t|\le T}\|f_t\|_{L^{q,\infty}(d\mu_t)}\ \le\ C(\alpha,q)\,\exp\!\Big(C(\alpha,q)\big(1+T^{2/\alpha}\big)\Big).
		\end{equation}
		We use the standard characterization (\cite{Grafa}: Page 11, Exercise 1.1.12):
		\begin{equation}\label{weaknorm}
			\|f_t\|_{L^{q,\infty}(d\mu_t)}\simeq \sup_{\substack{E\subset X\\ \mu_t(E)>0}}\ \frac{1}{\mu_t(E)^{1-\frac{1}{q}}}\int_E f_t\,d\mu_t.
		\end{equation}
		
		Fix a Borel set $E\subset X$ and set $y(t):=\mu_t(\Phi_t^{-1}(E))$. By \eqref{differential},
		\[
		y'(t)\le C_0\,r^{1-\alpha}\,y(t)^{1-\frac{1}{r}}\quad\Rightarrow\quad y(t)^{1/r}\le y(0)^{1/r}+C_0|t|\,r^{-\alpha}.
		\]
		Hence, for $|t|\le T$,
		\begin{equation}\label{Differentialine}
			\mu_t(\Phi_t^{-1}(E))\le\Big(\mu_0(E)^{1/r}+C_0T\,r^{-\alpha}\Big)^r.
		\end{equation}
		Using \eqref{changeweight} with $p=2q$ and Hölder,
		\[
		\mu_0(E)\le \Big\|\frac{d\mu_0}{d\mu_t}\Big\|_{L^{2q}(d\mu_t)}\,\mu_t(E)^{1-\frac{1}{2q}}\le M_{2q}\,\mu_t(E)^{1-\frac{1}{2q}}.
		\]
		Plugging this into \eqref{Differentialine},
		\[
		\mu_t(\Phi_t^{-1}(E))\le\Big(M_{2q}\,\mu_t(E)^{\frac{1-\frac{1}{2q}}{r}}+C_0T\,r^{-\alpha}\Big)^r.
		\]
		Now, by \eqref{weaknorm},
		\begin{align}
			\|f_t\|_{L^{q,\infty}(d\mu_t)}
			&\le \sup_{0<\mu_t(E)\le e^{-b_0}}\ \mu_t(E)^{-\big(1-\frac{1}{q}\big)}\mu_t(\Phi_t^{-1}(E))\nonumber\\
			&\quad +\ \sup_{\mu_t(E)\ge e^{-b_0}}\ \mu_t(E)^{-\big(1-\frac{1}{q}\big)}\mu_t(\Phi_t^{-1}(E)).\label{split}
		\end{align}
		For the first term, writing $\mu_t(E)=e^{-b}$ and optimizing in $r>1$,
		\[
		\sup_{b\ge b_0}\ \exp\!\Big(r\log M - \tfrac{b}{2q}+ T r^{1-\alpha}e^{\frac{b}{r}\big(1-\frac{1}{2q}\big)}\Big),
		\]
		with $M:=\max\{C_0,M_{2q}\}$. Choosing $r=\frac{b}{\log\log b}$ and taking $b_0$ large enough so that
		\[
		\frac{\log b}{(\log\log b)^{1-\alpha}}\le b^{\alpha/2},\quad
		b\ge \Big(\tfrac{T}{4q}\Big)^{2/\alpha},\quad
		b\ge \exp\exp\Big(\tfrac{\log M}{4q}\Big),
		\]
		one obtains that the exponent is $\le 0$ for all $b\ge b_0$, so this contribution is $\lesssim 1$.
		
		For the second term in \eqref{split}, take $r=2$ in \eqref{Differentialine} to get
		\[
		\mu_t(\Phi_t^{-1}(E))\le \Big(M_{2q}\,\mu_t(E)^{\frac{1-\frac{1}{2q}}{2}}+C_0T\,2^{-\alpha}\Big)^2.
		\]
		Therefore, for all $\mu_t(E)\ge e^{-b_0}$,
		\[
		\mu_t(E)^{-\big(1-\frac{1}{q}\big)}\mu_t(\Phi_t^{-1}(E))\ \lesssim\ e^{b_0(1-\frac{1}{q})}\,\Big(M_{2q}^2\,\mu_t(X)^{1-\frac{1}{2q}}+C_0^2T^2\Big).
		\]
		Choosing $b_0\simeq c_\alpha \big(1+T^{2/\alpha}\big)$ as above yields \eqref{weakLqbound}. Finally, interpolate between weak–$L^q$ and $L^1$ (note $\|f_t\|_{L^1(d\mu_t)}=\mu_t(X)$) to obtain \eqref{Lp:ftu}. Using \eqref{Lp:ftu} and Hölder gives \eqref{recursive}.
	\end{proof}


\end{document}